\documentclass{elsarticle}

\biboptions{sort&compress,numbers}
\usepackage[nodots]{numcompress}
\usepackage{lineno,url}
\usepackage{amsmath,amsfonts,amssymb}
\usepackage{amssymb,amsmath,amsthm}
\usepackage{overpic}
\usepackage{contour}\contourlength{0.7pt}
\usepackage{tikz}\usetikzlibrary{shapes,arrows}
\usepackage{versions}

\newtheoremstyle{mytheorem}{5pt plus 5pt minus 3pt}{4pt plus 3pt minus 1.5pt}
	{\itshape}{}{\bfseries}{.}{1ex plus 1ex minus .5ex}{}
\newtheoremstyle{mydef}{5pt plus 5pt minus 3pt}{4pt plus 3pt minus 1.5pt}
	{}{0pt}{\bfseries}{.}{1ex plus 1ex minus .5ex}{}
\newtheoremstyle{myremark}{5pt plus 5pt minus 3pt}{4pt plus 3pt minus 1.5pt}
	{}{0pt}{\itshape}{.}{1ex plus 1ex minus .5ex}{}
\theoremstyle{mytheorem}
\newtheorem{prop}{Proposition}[section]
\newtheorem{lemma}{Lemma}[section]
\newtheorem{theorem}{Theorem}[section]
\theoremstyle{mydef}
\newtheorem{dfn}[prop]{Definition}
\theoremstyle{myremark}

\newtheorem{rem}{Remark}
\newtheorem{cor}{Corollary}

\newcommand{\D}{\mathbf D}

\def\Acaption#1#2{\caption{#2}\vspace*{-#1}}

\definecolor{Mgreen}{RGB}{34,139,34}
\definecolor{blau}{rgb}{0.15,0.2,0.5}
\definecolor{gray}{rgb}{0.5,0.5,0.5}
\definecolor{drot}{rgb}{0.7,0,0.1}
\definecolor{gelb}{rgb}{.55,.40,.1}
\definecolor{magenta}{rgb}{1.,0.,1.}
\definecolor{cyan}{rgb}{0.,1.,1.}
\definecolor{green}{rgb}{0.,1.,0.}

\journal{JCAM}

\begin{document}

\begin{frontmatter}
\title{Explicit Gaussian quadrature rules for cubic splines with non-uniform knot sequences}
\author[VCC]{Rachid Ait-Haddou\corref{cor1}}
\ead{rachid.aithaddou@kaust.edu.sa}
\author[NumPor]{Michael Barto\v{n}}
\ead{Michael.Barton@kaust.edu.sa}
\author[NumPor]{Victor Manuel Calo}
\ead{Victor.Calo@kaust.edu.sa}

\cortext[cor1]{Corresponding author}

\address[VCC]{Visual Computing Center, King Abdullah University of Science and Technology, Thuwal 23955-6900, Kingdom of Saudi Arabia}
\address[NumPor]{Numerical Porous Media Center, King Abdullah University of Science and Technology, Thuwal 23955-6900, Kingdom of Saudi Arabia}

\begin{abstract}
We provide explicit expressions for quadrature rules
on the space of $C^1$ cubic splines with non-uniform, symmetrically stretched knot sequences.
The quadrature nodes and weights are derived via an explicit recursion
that avoids an intervention of any numerical solver and the rule is optimal, that is,
it requires minimal number of nodes. Numerical experiments validating
the theoretical results and the error estimates of the quadrature rules are also presented.
\end{abstract}


\end{frontmatter}

\section{Introduction}\label{intro}
The problem of numerical quadrature has been of interest for decades due to its
wide applicability in many fields spanning collocation methods \cite{Sloan-1988},
integral equations \cite{Atkinson-1976}, finite elements methods \cite{Solin-2003}
and most recently, isogeometric analysis \cite{Cottrell-2009}. 
Computationally, the integration of a function
is an expensive procedure and quadrature turned out to be a cheap, robust and elegant alternative.

A quadrature rule, or shortly a quadrature, is said to be an \textit{$m$-point rule},
if $m$ evaluations of a function $f$ are needed to approximate its weighted integral
over an interval $[a,b]$
\begin{equation}\label{eq:GaussQuad}
\int_a^b \omega(x) f(x) \, \mathrm{d}x = \sum_{i=1}^{m} \omega_i f(\tau_i) + R_{m}(f),
\end{equation}
where $\omega$ is a fixed non-negative \emph{weight function} defined over $[a,b]$.
Typically, the rule is required to be \emph{exact}, that is, $R_m(f) \equiv 0$
for each element of a predefined linear function space $\mathcal{L}$. In the case
when $\mathcal{L}$ is the linear space of polynomials of degree at most $2m-1$,
then the $m$-point Gaussian quadrature rule \cite{Gautschi-1997} provides the \emph{optimal}
rule that is exact for each element of $\mathcal{L}$, i.e. $m$ is the minimal number
of \textit{nodes} at which $f$ has to be evaluated.
The Gaussian nodes are the roots of the orthogonal polynomial $\pi_{m}$ where
$(\pi_0,\pi_1,\ldots,\pi_m,\ldots )$ is the sequence of orthogonal polynomials with
respect to the measure $\mu(x) = \omega(x)dx$. Typically, the nodes of the
Gaussian quadrature rule
are computed numerically using for example, the Golub-Welsh algorithm \cite{Golub-1969},
in the case the three-term recurrence relations for the orthogonal polynomials can
be expressed.

In the case when $\mathcal{L}$ is a Chebyshev space of dimension $2m$, Studden and Karlin
proved the \emph{existence} and \emph{uniqueness} of optimal $m$-point generalized quadrature rules,
which due to optimality are also called Gaussian,
that are exact for each element of the space $\mathcal{L}$ \cite{Karlin-1966}.
The nodes and weights of the quadrature rule can be computed using numerical schemes
based on Newton methods \cite{Ma-1996}.

In the case when $\mathcal{L}$ is a linear
space of splines, a favourite alternative to polynomials
due to their approximation superiority and the inherent locality property
\cite{deBoor-1972,Elber-2001,Hoschek-2002-CAGD},
Micchelli and Pinkus \cite{Micchelli-1977} derived the optimal
number of quadrature nodes. Moreover, the range of intervals,
the knot sequence subintervals that contain at least one node, was specified.
Their formula preserves the ``double precision'' of Gaussian rules for polynomials, that is,
for a spline function with $r$ (simple) knots, asymptotically, the number
of nodes is $[\frac{r}{2}]$. Whereas the optimal quadrature rule is unique in
the polynomial case and the Chebyshev systems case, this is in general not true
for splines. The computation of the nodes and weights of the optimal spline
quadrature (Gaussian quadrature) is rather a challenging problem as
the non-linear systems the nodes and weights satisfy depend on truncated power
functions. The systems become algebraic only with the right guess of the
knot intervals where the nodes lie.


Regarding the optimal quadrature rules for splines, the quadrature schemes
differ depending on the mutual relation between the \emph{degree} and \emph{continuity} $(d,c)$.
For cases with lower continuity, a higher number of nodes is required for the optimal
quadrature rule.
Also, the choice of the domain can bring a significant simplification.
Whereas an exact quadrature rule -- when the weight function $\omega \equiv 1$ in Eq.~(\ref{eq:GaussQuad}) --
can be obtained by simply evaluating $f$ at every second knot (midpoint)
for uniform splines of even (odd) degree over \textit{a real line}
\cite{Hughes-2010}, a closed interval is an obstacle, even for uniform splines,
that can be resolved only by employing numerical solvers \cite{Hughes-2012}.

Thus, the insightful proposition of Nikolov \cite{Nikolov-1996}, which yield optimal
and explicit quadrature rules for $(3,1)$ \emph{uniform} splines (with $\omega \equiv 1$),
is surprising. In Nikolov's scheme, a recursive relation between
the neighboring nodes is derived and, since the resulting system is of cubical degree,
a closed form formula is given to iteratively compute the nodes and weights.

In this paper, we generalize the quadrature rules of \cite{Nikolov-1996}
for splines with certain \textit{non-uniform} knot sequences,
keeping the desired properties of explicitness, exactness and optimality.
The rest of the paper is organized as follows.
In Section~\ref{sec:quad}, we recall some basic properties of $(3,1)$ splines
and derive their Gaussian quadrature rules. In Section~\ref{sec:err}, the error estimates are given
and Section~\ref{sec:ex} shows the numerical experiments. Finally,
possible extensions of our method are discussed in Section~\ref{sec:con}.

\section{Gaussian quadrature formulae for $C^1$ cubic splines}\label{sec:quad}

 \begin{figure*}
 \hfill
 \begin{overpic}[width=.69\columnwidth,angle=0]{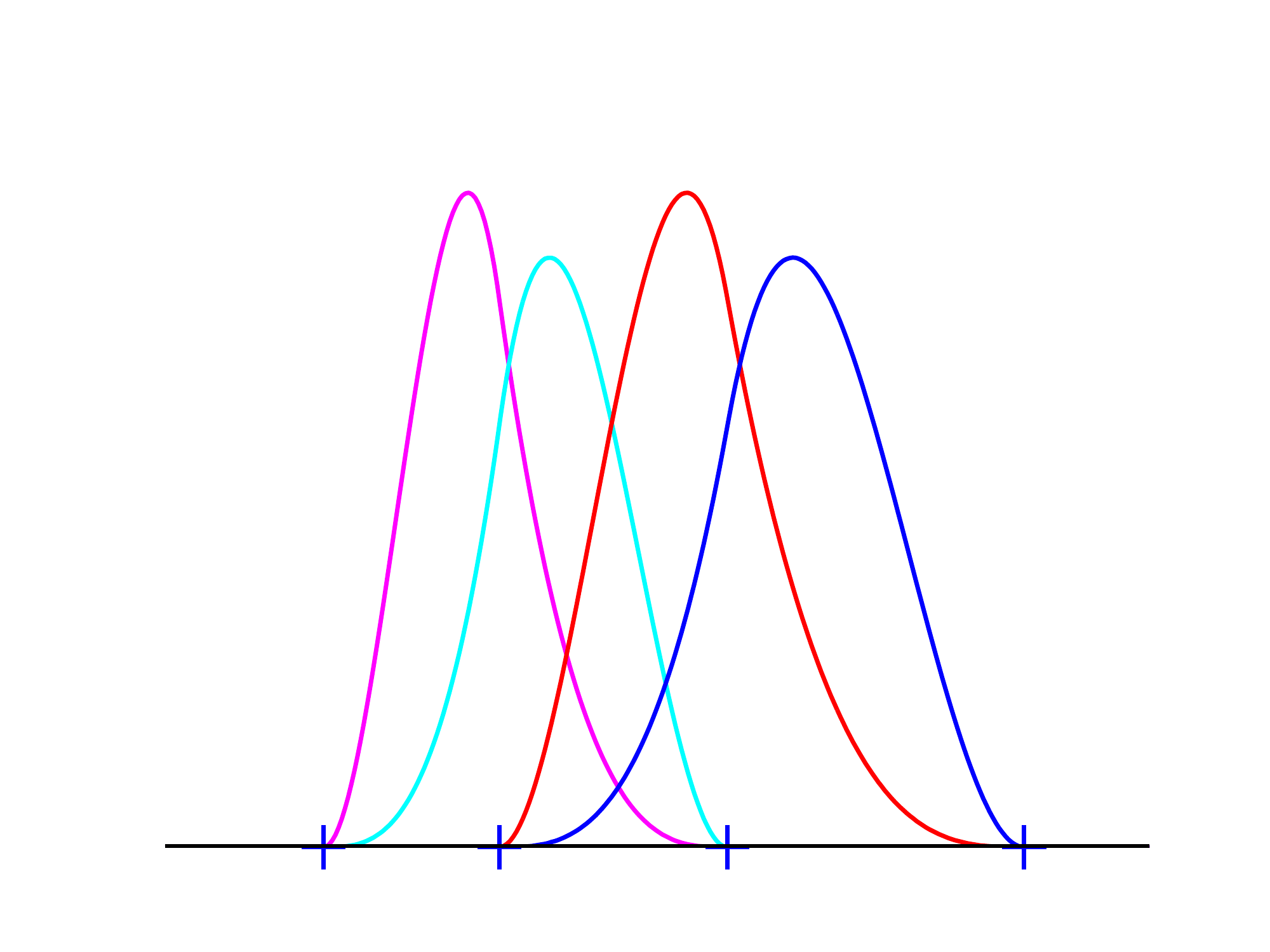}
    \put(20,3){$x_{k-2}$}
    \put(36,3){$x_{k-1}$}
    \put(55,3){$x_{k}$}
    \put(75,3){$x_{k+1}$}
    \put(20,40){\color{magenta}{$D_{2k-1}$}}
    \put(42,58){\color{cyan}{$D_{2k}$}}
  \end{overpic}
 \hfill
 \vspace{-4pt}
 \Acaption{1ex}{Four consecutive knots $x_{k-2}, \dots, x_{k+1}$ of a stretched knot sequence, each of multiplicity two. Four spline
 basis functions with non-zero support on $[x_{k-1},x_{k}]$ are displayed.}\label{fig:C1Splines}
 \end{figure*}

In this section we recall few basic properties of $(3,1)$ splines and derive
explicit formulae for computing quadrature nodes and weights for a particular family of knot sequences.
Throughout the paper, $\pi_{n}$ denotes the linear space of polynomials
of degree at most $n$ and $[a,b]$ is a non-trivial real compact interval.

\subsection{$C^1$ cubic splines with symmetrically stretched knot sequences}\label{ssec:spline}

We start with the definition of the particular knot sequences above which the spline spaces are built.

\begin{dfn}\label{def:stretch}
A finite sequence $\mathcal{X}_n = (a=x_0,x_1,...,x_{n-1},x_{n} = b)$ of pairwise
distinct real numbers in the interval $[a,b]$ is said to be a
{\bf{symmetrically stretched knot sequence}} if the sequence is symmetric
with respect to the midpoint of the interval $[a,b]$ and such that
\begin{equation}\label{strech}
 x_k - 2 x_{k+1} +  x_{k+2} \geq 0 \quad \textnormal{for}
\quad k=0,...,[\frac{n}{2}]-1.
\end{equation}
\end{dfn}

Denote by $S^{n}_{3,1}$ the linear space of $C^1$ cubic splines over a symmetrically stretched
knot sequence $\mathcal{X}_n =(a=x_0,x_1,...,x_n=b)$
\begin{equation}\label{eq:family}
S^{n}_{3,1} = \{ f\in C^{1}[a,b]: f|_{(x_k,x_{k+1})} \in \pi_3, k=0,...,n-1\}.
\end{equation}
The dimension of the space $S^{n}_{3,1}$ is $2n+2$.
\begin{rem}
In the B-spline literature \cite{deBoor-1972,Hoschek-2002-CAGD,Elber-2001}, the knot sequence is
usually written with knots' multiplicities. As in this paper the multiplicity is
always two at every knot, we omit the classical notation and, throughout the paper,
write $\mathcal{X}_n$ without multiplicity, i.e. $x_k<x_{k+1}$, $k=0,\dots, n-1$.
\end{rem}

Similarly to \cite{Nikolov-1996}, we find it convenient to work with the
non-normalized B-spline basis. To define the basis, we extend our knot sequence
$\mathcal{X}_n$ with two extra knots outside the interval $[a,b]$ that we set to be
\begin{equation}\label{boundary_condition}
x_{-1} = 2x_{0}-x_{1}
\quad \textnormal{and} \quad
x_{n+1} =2 x_{n} - x_{n-1}.
\end{equation}
Note that the choice of $x_{-1}$ and $x_{n+1}$ is to get particular integrals in (\ref{boundaryIntegral})
that simplify expressions in Section~\ref{ssec:quad}. We emphasize that this setting does not affect the quadrature
rule derived later in Theorem~\ref{theo:quad}.
Denote by $\D = \{D_i\}_i^{2n+2}$ the basis of $S^{n}_{3,1}$ where
\begin{equation*}
\begin{split}
&D_{2k-1}(t) = [x_{k-2},x_{k-2},x_{k-1},x_{k-1},x_k](. - t)_{+}^{3}   \\
&D_{2k}(t) = [x_{k-2},x_{k-1},x_{k-1},x_{k},x_k](. - t)_{+}^{3}, \\
\end{split}
\end{equation*}
where $[.]f$ stands for the divided difference and $u_{+} = \max(u,0)$ is the truncated
power function, see Fig.~\ref{fig:C1Splines}. Among the basic properties of the
basis $\D$, we need to recall the fact that for any $k=1,2,\ldots,n+1$, $D_{2k-1}$
and $D_{2k}$ have the same support, that is,
$\textnormal{supp}(D_{2k-1}) = \textnormal{supp}(D_{2k}) = [x_{k-2}, x_{k}]$, and
$D_{2k-1} (t) > 0$, $D_{2k}(t) > 0$ for all $t \in (x_{k-2}, x_{k})$.
Moreover, for $k=3,...,2n$, we have
\begin{equation}\label{interiorIntegral}
I[D_k] = \frac{1}{4}\; \textnormal{for} \quad k = 3,4,\ldots,2n,
\end{equation}
where $I[f]$ stands for the integral of $f$ over the interval $[a,b]$.
With the choice made in (\ref{boundary_condition}), we have
\begin{equation}\label{boundaryIntegral}
I[D_1] =  I[D_{2n+2}] = \frac{1}{16}
\quad \textnormal{and} \quad
I[D_2] =  I[D_{2n+1}]= \frac{3}{16}.
\end{equation}
Using the standard definition of divided difference for multiple knots,
explicit expressions for $D_{2k-1}(t)$
and $D_{2k}(t)$ with $t \in [x_{k-2},x_{k}]$ are obtained as
\begin{equation*}
D_{2k-1}(t) = a_{k} (x_k - t)_{+}^{3} + b_{k}  (x_{k-1} - t)_{+}^{3} + c_{k}  (x_{k-1} - t)_{+}^{2},
\end{equation*}
where, setting $h_k = x_{k} - x_{k-1}$ for $ k=0,1,\ldots,n+1$,
\begin{equation*}
 a_k = \frac{1}{h_k^{2} (h_k +h_{k-1})^{2}},\;
 b_k = \frac{2 h_{k} - h_{k-1}}{h_{k-1}^3 h_k^2},\;
c_k = \frac{-3}{h_{k-1}^2  h_{k}}.
\end{equation*}
Similarly, we obtain
\begin{equation*}
D_{2k}(t) = \alpha_k  (x_k - t)_{+}^{3} + \beta_k  (x_k - t)_{+}^{2} +
\gamma_k  (x_{k-1} - t)_{+}^{3} + \eta_k  (x_{k-1} - t)_{+}^{2},
\end{equation*}
where
\begin{equation*}
\alpha_k = \frac{-3 h_{k} - 2 h_{k-1}}{(h_k + h_{k-1})^2 h_k ^3},\;
\beta_k = \frac{3}{(h_k + h_{k-1}) h_k^2},\;
\gamma_k = \frac{2 h_{k-1} -  h_{k}}{h_{k-1}^2 h_{k}^3},\;
\eta_k =  \frac{3}{h_{k-1} h_{k}^2}.
\end{equation*}

That is, $D_{2k-1}$ and $D_{2k}$, are expressed by three parameters $x_{k-2}$, $x_{k-1}$ and $x_k$,
due to the fact that $[x_{k-2},x_k]$ is the maximal interval where both have a non-zero support,
see Fig.~\ref{fig:C1Splines}. Moreover, we have the following:

 \begin{figure*}
 \hfill
 \begin{overpic}[width=.69\columnwidth,angle=0]{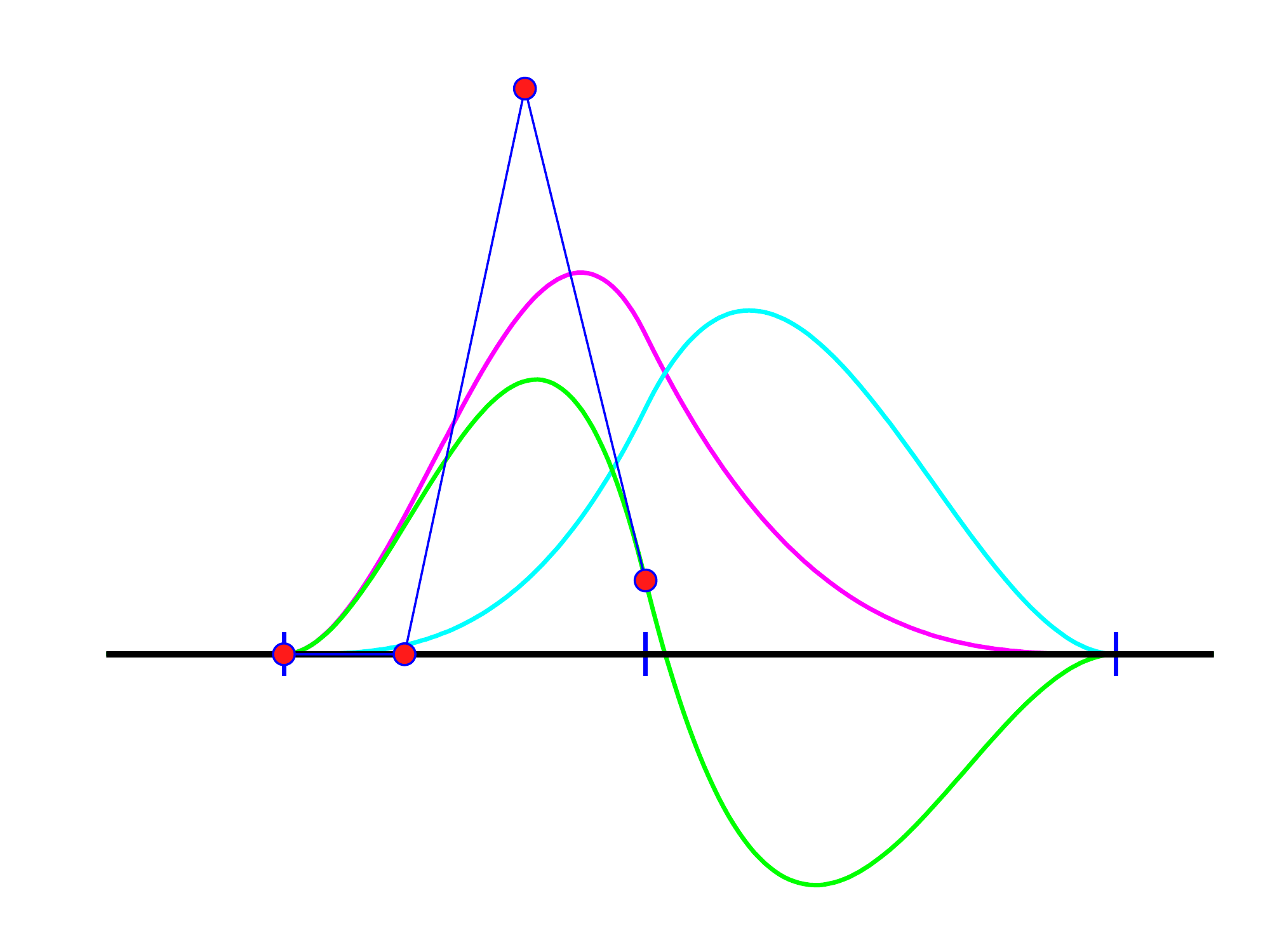}
    \put(18,18){$x_{k-2}$}
    \put(50,18){$x_{k-1}$}
    \put(90,18){$x_{k}$}
    \put(20,43){\color{magenta}{$D_{2k-1}$}}
    \put(65,50){\color{cyan}{$D_{2k}$}}
    \put(75,10){\color{green}{$Q$}}
    \put(43,67){$[\frac{x_{k-2}+2x_{k-1}}{3},q_2]$}
    \put(53,30){$[x_{k-1},q_3]$}
	\end{overpic}
 \hfill
 \vspace{-10pt}
 \Acaption{1ex}{The stretching property of the knot sequence, $x_k - x_{k-1} \geq x_{k-1} - x_{k-2}$, guarantees
 non-negativity of $D_{2k-1} - D_{2k}$ on $[x_{k-2}, x_{k-1}]$. Representing their difference, $Q$, as a B\'ezier curve
 on $[x_{k-2}, x_{k-1}]$, all its control points (red) have non-negative $y-$coordinates.}\label{fig:Lemma}
 \end{figure*}
\begin{lemma}\label{lemmaD}
Let $\mathcal{X}_{n} = (a=x_0,x_1,...,x_n=b)$ be a symmetrically stretched knot sequence.
Then for any $k = 2,...,[n/2]+1$
\begin{equation*}
D_{2k-1}(t) > D_{2k}(t)\quad  \textnormal{for any} \quad t \in  (x_{k-2},x_{k-1}).
\end{equation*}
\end{lemma}
\begin{proof}
Over the interval $(x_{k-2},x_{k-1})$,  the function $Q = D_{2k-1} - D_{2k}$ is a single
cubic polynomial. Therefore, it can be expressed in terms of Bernstein basis and can be
viewed as a B\'ezier curve on $(x_{k-2},x_{k-1})$, see Fig.~\ref{fig:Lemma},
$$
Q(t) = \sum_{i=0}^{3} q_i B_i^3(t), \; \textnormal{where} \;
B_i^3(t) = \binom{3}{i} \left(\frac{t-x_{k-2}}{x_{k-1} - x_{k-2}}\right)^i
\left(\frac{x_{k-1} - t}{x_{k-1} - x_{k-2}}\right)^{3-i}.
$$
Straightforward computations of the control points $(q_0,q_1,q_2,q_3)$ of $Q$
over the interval $[x_{k-2},x_{k-1}]$ leads to
$$
(q_0,q_1,q_2,q_3) = \left(0,0, \frac{1}{x_{k} - x_{k-2}},\frac{x_{k} -
2 x_{k-1} + x_{k-2}}{(x_{k} - x_{k-2})^2}\right).
$$
Therefore, according to (\ref{strech}), the control points are nonnegative,
with the third control point $q_2$ strictly positive. Therefore, $Q$ can only vanish
at $x_{k-2}$ and $x_{k-1}$ and is strictly positive over $(x_{k-2},x_{k-1})$.
\end{proof}

\subsection{Gaussian quadrature formulae}\label{ssec:quad}
In this section, we derive a quadrature rule for the family $S^{n}_{3,1}$, see (\ref{eq:family}),
and show it meets the three desired criteria, that is, the rule is optimal, exact and explicit.
With respect to exactness, according to \cite{Micchelli-1977,Micchelli-1972} there exists
a quadrature rule
\begin{equation}\label{quadrature}
I(f) = \int_{a}^{b} f(t) dt  \simeq I_{n+1}(f) = \sum_{i=1}^{n+1} \omega_i f(\tau_i)
\end{equation}
that is exact for every function $f$ from the space $S^{n}_{3,1}$.
The explicitness and optimality follow from the construction.

%
%
\begin{lemma}\label{lemmag1}
Let $\mathcal{X}_{n} = (a=x_0,x_1,...,x_n=b)$ be a symmetrically stretched knot sequence.
Each of the intervals $I_{k} = (x_{k-1},x_{k})\; (k=1,...,[n/2])$ contains at least one node of
the Gaussian quadrature rule (\ref{quadrature}).
\end{lemma}
\begin{proof}
We proceed by induction on the index of the segment $I_{k}$.
There must be a node of the Gaussian quadrature rule in the interval $I_{1}$,
otherwise, using the exactness of the quadrature rule for $D_{1}$, we obtain
$I(D_{1}) = 0$ which contradicts equalities (\ref{boundaryIntegral}).
Now, let us assume that every segment $I_{l}$ contains -- one or several -- Gaussian nodes
for $l=1,2,...,k-1$. If the interval $I_{k}$ has no Gaussian nodes,
then using Lemma \ref{lemmaD}, we arrive to the following contradiction
$$
\frac{1}{4} = I[D_{2k}] = \sum_{\tau_j \in I_{k-1}} \omega_j D_{2k}(\tau_{j}) <
 \sum_{\tau_j \in I_{k-1}} \omega_j D_{2k-1}(\tau_{j}) =  I[D_{2k-1}] = \frac{1}{4}.
$$
\end{proof}

\begin{cor}\label{corollary1}
If $n$ is an even integer, then each of the intervals $I_{k} = (x_{k-1},x_{k})$ $(k=1,2,\ldots,n)$
contains exactly one Gaussian node and the middle $x_{n/2} = (a+b)/2$ of the interval $[a,b]$
is also a Gaussian node. If $n$ is odd then each of the intervals $I_{k} = (x_{k-1},x_{k})$
$(k=1,2,\ldots,n; k\not= (n+1)/2)$ contain exactly one Gaussian node, while the interval
$I_{(n+1)/2}$ contains two Gaussian nodes, positioned symmetrically with respect to $(a+b)/2$.
\end{cor}
\begin{proof}
If $n$ is an even number then by symmetry, we obtain at least one Gaussian node
in each interval $I_{k}$ for $k=1,2,\ldots,n$. If one of the intervals $I_{k}$ has more than
one node then by symmetry, we get more than $n+2$ nodes for the quadrature, contradicting our
quadrature rule (\ref{quadrature}). Moreover,  by virtue of symmetry, the last missing Gaussian node
is forced to be the middle of the interval. Now, if $n$ is an odd integer, then by symmetry, each of the intervals
$I_{k}$, $k=1,2,\ldots,n$ contains at least one Gaussian node. Let us assume that the middle
interval $I_{(n+1)/2}$ contains exactly one node, then at least one of the remaining intervals
contains two nodes. By symmetry, the number of nodes will be at least $(n+2)$, contradicting our quadrature
rule (\ref{quadrature}). Therefore, the middle interval $I_{(n+1)/2}$ contains exactly two nodes while each
of the remaining intervals contain exactly one Gaussian node of the quadrature rule (\ref{quadrature}).
\end{proof}

Throughout the rest of this work, we use the following notation:  For $k=1,2,\ldots,[n/2]+1$, we set
\begin{equation}
\theta_k = x_k - \tau_k; \quad \rho_k = x_{k+1} - \tau_k \quad \textnormal{and}
\end{equation}
\begin{equation*}
\begin{split}
& A_k = \frac{1}{4} - \omega_k \left(a_{k+1} \rho_k^3 + b_{k+1} \theta_k^3 + c_{k+1} \theta_k^2 \right), \\
&B_k =  \frac{1}{4} - \omega_k \left(\alpha_{k+1} \rho_k^3 + \beta_{k+1} \rho_k^2 +
\gamma_{k+1} \theta_k^3 + \eta_{k+1} \theta_k^2 \right).
\end{split}
\end{equation*}
The explicit representation of the B-spline basis $D_i$ gives
\begin{equation}\label{equations_tau}
\begin{split}
& D_{2k-1}(\tau_k) = a_k \theta_k^3, \\
& D_{2k}(\tau_k) = \alpha_k \theta_k^3 + \beta_k \theta_k^2, \\
& D_{2k+1}(\tau_k) = a_{k+1} \rho_k^3 + b_{k+1} \theta_k^3 + c_{k+1}\theta_k^2, \\
& D_{2k+2}(\tau_k) = \alpha_{k+1} \rho_k^3 + \beta_{k+1} \rho_k^2 + \gamma_{k+1}\theta_k^3 + \eta_{k+1}\theta_k^2.
\end{split}
\end{equation}
We are ready now to proceed with the recursive algorithm which
starts at the domain's first subinterval $[x_0,x_1]$ by computing the first
node and weight, and sequentially parses the subintervals towards the domain's
midpoint, giving explicit formulae for the remaining unknowns $\tau_i$, $\omega_i$, $i=2,\dots,[n/2]+1$.
There is, according to Corollary \ref{corollary1}, a unique Gaussian node in the
interval $(x_0,x_1)$. This node is obtained by solving the system
\begin{equation*}
\begin{split}
& I[D_1] = \omega_1 D_1(\tau_1) = \frac{1}{16} = \omega_1 a_1 \theta_1^3, \\
& I[D_2] = \omega_1 D_2(\tau_1) = \frac{3}{16} = \omega_1 (\alpha_1 \theta_1^3 + \beta_1 \theta_1^2), \\
\end{split}
\end{equation*}
leading to the unique solution for $\theta_1$ and $\omega_1$ to be expressed as
\begin{equation*}
\theta_1 = \frac{\beta_1}{3 a_1 - \alpha_1} = \frac{3}{4}h_1
\quad \text{and} \quad
\omega_1 = \frac{1}{16 a_1 \theta_1^3}  = \frac{16}{27} h_1.
\end{equation*}
The remaining nodes and weights are computed in turn explicitly using the recipe
formalized as follows:

\begin{theorem}\label{theo:quad}
The sequence of nodes and  weights of the Gaussian quadrature rule (\ref{quadrature})
are given explicitly as $\theta_1 = \frac{3}{4}h_1,  \omega_1 =  \frac{16}{27} h_1$
and for $i =1,2,...,[n/2]-1$ by the recurrence relations
\begin{equation}\label{mainquadrature}
\theta_{i+1} = \frac{A_i \beta_{i+1}}{a_{i+1} B_i - \alpha_{i+1} A_i}
\quad \text{and} \quad
\omega_{i+1} = \frac{A_i}{a_{i+1} \theta_{i+1}^3}.
\end{equation}
If $n$  is even ($n=2m$) then $\tau_{m+1} = x_{m} = (a+b)/2$ and
\begin{equation}\label{even}
\omega_{m+1} = \frac{A_m + B_m - \frac{1}{4}}{a_{m+1} \theta_{m+1}^3}.
\end{equation}
If $n$ is odd ($n = 2m-1$) then $\theta_m$ is the greater root in $(0,x_m - x_{m-1})$
of the cubic equation
\begin{equation*}
\begin{split}
& \left( A_{m-1}(\alpha_m+b_{m+1}) - B_{m-1}(a_{m}+ \gamma_{m+1})\right) \theta_m^3 + \\
& \left( A_{m-1}(\beta_m+c_{m+1}) - B_{m-1}\eta_{m+1} \right) \theta_m^2 + \\
& (A_{m-1} a_{m+1} - B_{m-1} \alpha_{m+1})\rho_m^3 - B_{m-1} \beta_{m+1} \rho_m^2 =0,
\end{split}
\end{equation*}
and
\begin{equation*}
\omega_{m} =   \frac{A_{m-1}}{(\gamma_{m+1}+ a_m) \theta_m^3 +
\eta_{m+1} \theta_m^2 + \alpha_{m+1} \rho_{m}^3 + \beta_{m+1} \rho_m^2}.
\end{equation*}
\end{theorem}
\begin{proof}
The proof proceeds by induction. We assume $\theta_l, \omega_l$
known for $l=1,2,\ldots,k$ ( $k \leq [n/2]-2$). Using (\ref{equations_tau})
we compute $\theta_{k+1}$ and $\omega_{k+1}$ by solving the system
$ I[D_{2k+1}] = 1/4$ and $ I[D_{2k+2}] = 1/4$, that is
\begin{equation*}
\begin{split}
 \frac{1}{4} &= \omega_k D_{2k+1}(\tau_k) + \omega_{k+1} D_{2k+1}(\tau_{k+1})
 = (\frac{1}{4} - A_k) + \omega_{k+1} a_{k+1} \theta_{k+1}^3,  \\
\frac{1}{4}  & = \omega_k D_{2k+2}(\tau_k) + \omega_{k+1} D_{2k+2}(\tau_{k+1})
=  (\frac{1}{4} - B_k) + \omega_{k+1} (\alpha_{k+1} \theta_{k+1}^3  + \beta_{k+1} \theta_{k+1}^2).  \\
\end{split}
\end{equation*}
Eliminating $\omega_{k+1}$ leads to the recurrence relations (\ref{mainquadrature}).
If $n$ is even ($n = 2m$), then by Corollary \ref{corollary1} we have $\tau_{m+1} = (a+b)/2$.
To compute the associated weight $\omega_{m+1}$, we take into account the symmetry, which gives
$\omega_{m} = \omega_{m+2}$ and $D_{2m+1}(\tau_{m+2}) =D_{2m+2}(\tau_{m})$, and solve
\begin{equation*}
\frac{1}{4} = I[D_{2m+1}]  = \omega_m [D_{2m+1}(\tau_m) + D_{2m+2}(\tau_m)] +
\omega_{m+1} D_{2m+1}(\tau_{m+1}).
\end{equation*}
Using (\ref{equations_tau}), we obtain (\ref{even}). If $n$ is odd ($n = 2m-1$), then according to
Corollary~\ref{corollary1} the two nodes $\tau_{m}$ and $\tau_{m+1}$ belong to the interval $(x_{m-1},x_{m})$.
Due to the symmetry, we have $\omega_{m} = \omega_{m+1}$ and $\tau_{m+1} = (a+b)-\tau_{m}$ and
\begin{equation*}
D_{2m-1}(\tau_{m+1}) = D_{2m+2}(\tau_m), \quad  D_{2m}(\tau_{m+1}) = D_{2m+1}(\tau_m).
\end{equation*}
Using the exactness of the quadrature rule for $D_{2m-1}$ and $D_{2m}$, we obtain
\begin{equation*}
\begin{split}
& \omega_m a_m \theta_m^3 = A_{m-1} + B_{m} - \frac{1}{4} \\
& \omega_{m} (\alpha_m \theta_m^3 + \beta_m \theta_m^2) = A_{m} + B_{m-1} - \frac{1}{4}
\end{split}
\end{equation*}
Solving the above system for $\theta_m$ and $\omega_m$ proves the theorem.
\end{proof}

\section{Error estimation for the $C^1$ cubic splines quadrature rule}\label{sec:err}
In the previous section, we have derived a quadrature rule that exactly
integrates functions from $S^{n}_{3,1}$. If $f$ is not an element of $S^{n}_{3,1}$,
the rule produces a certain error, also called \emph{remainder}, and the analysis of this error
is the objective of this section.

Let $W_{1}^{r} = \{ f \in C^{r-1}[a,b]; \; f^{(r-1)} \textnormal{abs. cont.}, \;   ||f||_{L_1} < \infty \}$.
As the quadrature rule (\ref{quadrature}) is exact for polynomials of degree at most three,
for any element $f \in W_{1}^{d}$, $d \geq 4$, we have
\begin{equation*}
R_{n+1}[f] :=  I(f) - I_{n+1}(f) = \int_{a}^{b} K_4(R_{n+1};t) f^{(4)}(t) dt,
\end{equation*}
where the Peano kernel \cite{Gautschi-1997} is given by
\begin{equation*}
 K_4(R_{n+1};t) = R_{n+1} \left[ \frac{(t-.)_{+}^{3}}{3!} \right].
\end{equation*}
An explicit representation for the Peano kernel over the interval $[a,b]$ in terms of the
weights and nodes of the quadrature rule (\ref{quadrature})  is given by
\begin{equation}
 K_4(R_{n+1};t) = \frac{(t-a)^4}{24} - \frac{1}{6} \sum_{k=1}^{n+1} \omega_k (t - \tau_k)_{+}^{3}.
\end{equation}
Moreover, according to a general result for monosplines and quadrature rules \cite{Micchelli-1977}, the only zeros
of the Peano kernel over $(a,b)$ are the double knots of the cubic spline, see Section~\ref{sec:ex} in particular
Fig.~\ref{fig:PK} for an illustration. Therefore, for any $t \in (a,b)$,
$K_4(R_{n+1};t) \geq 0$ and, by the mean value theorem, there exists a real number $\xi \in [a,b]$ such that
\begin{equation}\label{remainder}
R_{n+1}(f) = c_{n+1,4}  f^{(4)}(\xi) \quad  \textnormal{with } \quad  c_{n+1,4} =  \int_{a}^{b} K_{4}(R_{n+1}; t) dt.
\end{equation}
Hence, the constant $c_{n+1,4}$ of the remainder $R_{n+1}$ is always positive and our quadrature rule belongs to the
family of positive definite quadratures of order $4$, e.g., see \cite{Nikolov-1996, Nikolov-1995, Schmeisser-1972}.
To compute the constant $c_{n+1,4}$, we can follow the approach of \cite{Nikolov-1996} by expressing the exactness
of our quadrature rule for the truncated powers $(x_k-t)_{+}^{2}, (x_k-t)_{+}^{3}; k=0,1,...,n $.
As the symmetric stretched knot sequences satisfy the assumptions of Theorem 2.2 in \cite{Nikolov-1996}, the proof applies straightforwardly
to our non-uniform setting, and the constant of the remainder is expressed as
\begin{theorem}
The error constant $c_{n+1,4}$ of the quadrature rule (\ref{quadrature}) is given by
\begin{equation}\label{eq:c}
c_{n+1,4} =  \frac{1}{720} \sum_{k=0}^{[(n+1)/2]} (x_{k+1}-x_{k})^5 -
\frac{1}{12} \sum_{k=1}^{[(n+1)/2]} \omega_k (x_{k-1} - \tau_{k})^2 (x_{k} - \tau_{k})^2.
\end{equation}
\end{theorem}


\section{Numerical Experiments}\label{sec:ex}

We applied the quadrature rule to various symmetrically stretched knot sequences;
the nodes and weights computed by our formulae are summarized in Table~\ref{tabW}.
Even though the space of admissible stretched knot sequences is infinite dimensional,
for the sake simplicity, the proposed quadrature rule was tested on those
that are determined by the fewest possible number of parameters.

\begin{figure}[tbh!]
\hfill
 \begin{overpic}[width=.49\columnwidth,angle=0]{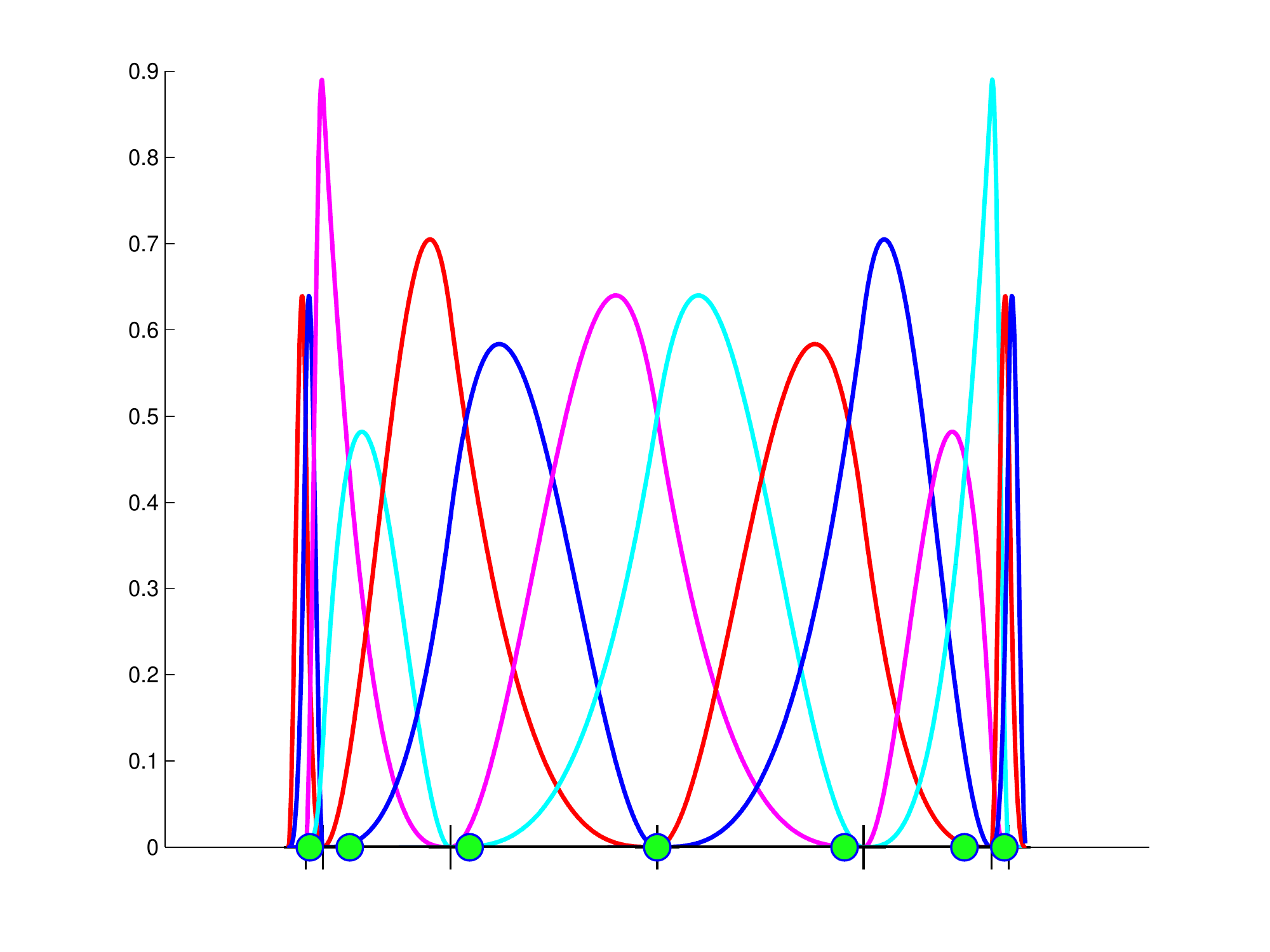}
    \put(10,2){$x_0=-1$}
    \put(45,2){$x_3=\tau_4$}
    \put(72,2){$\tau_{6}$}
    \put(80,2){$x_6=1$}
	\end{overpic}
 \hfill
 \begin{overpic}[width=.49\columnwidth,angle=0]{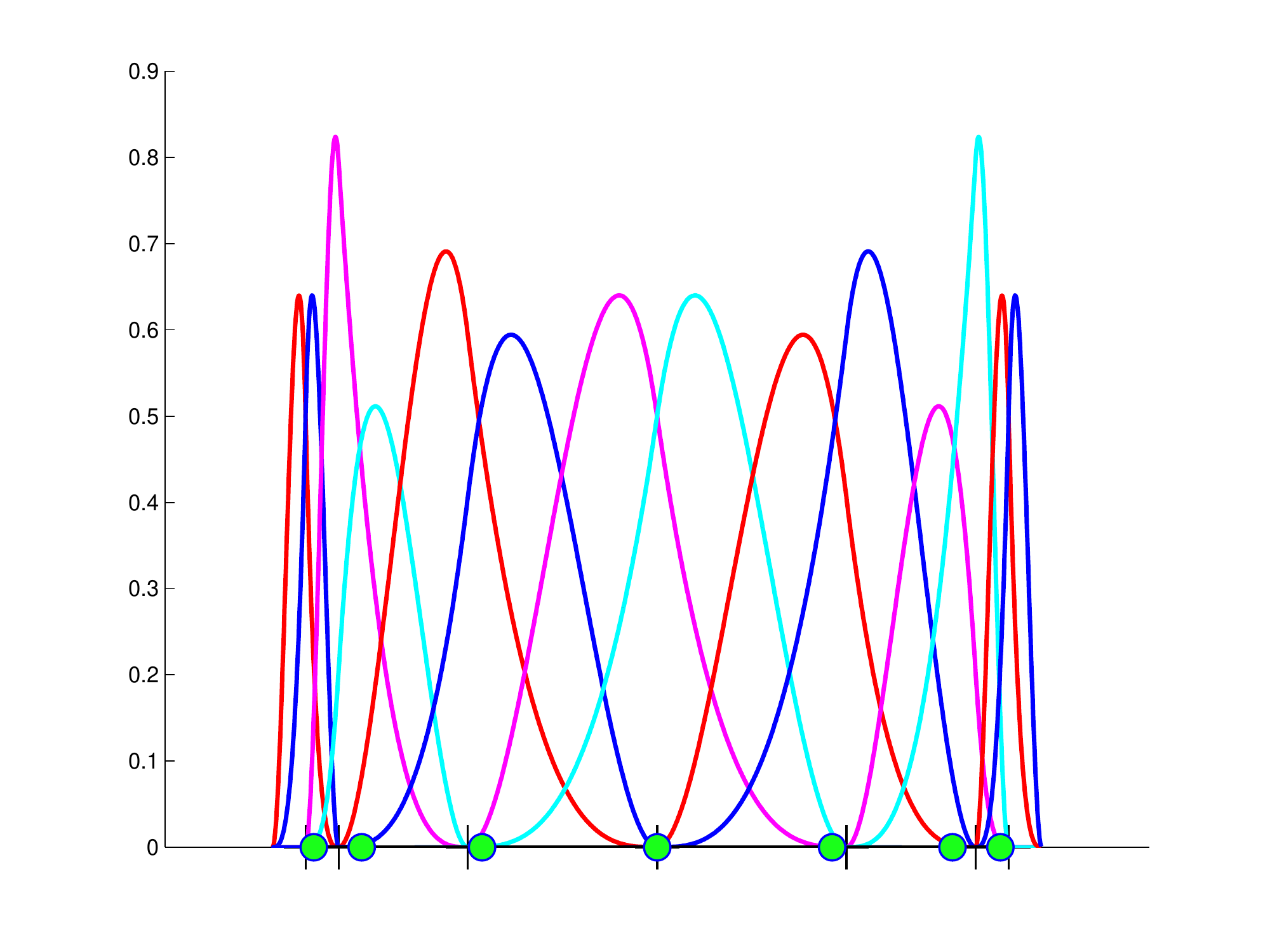}
    \put(15,2){$x_0=0$}
    \put(45,2){$x_3=\tau_4$}
    \put(72,2){$\tau_{6}$}
    \put(80,2){$x_6=1$}
	\end{overpic}
 \hfill
\vspace{-0.01cm}
 \hfill
 \begin{overpic}[width=.49\columnwidth,angle=0]{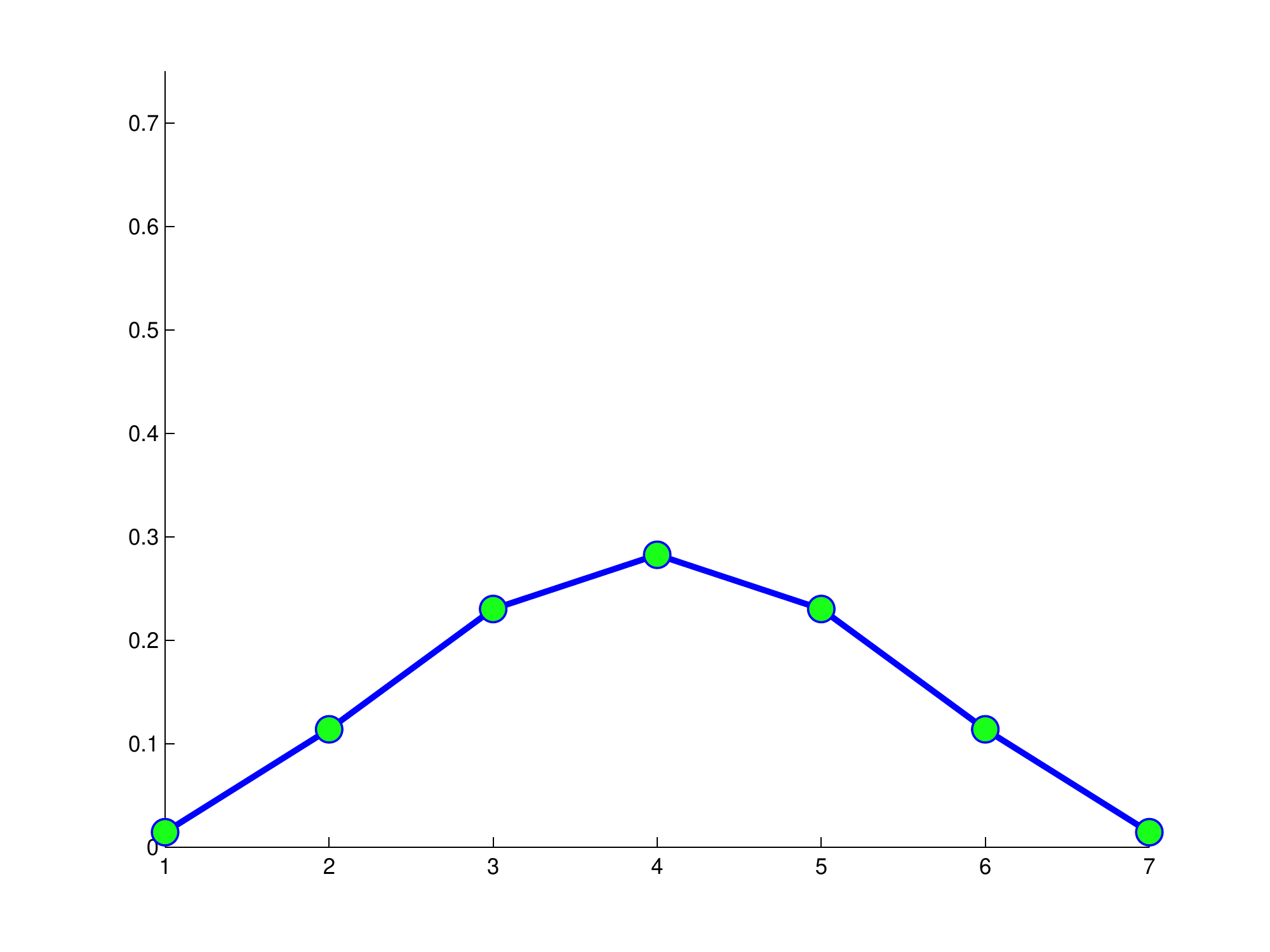}
    \put(38,21){$w_3$}
	\end{overpic}
 \hfill
 \begin{overpic}[width=.49\columnwidth,angle=0]{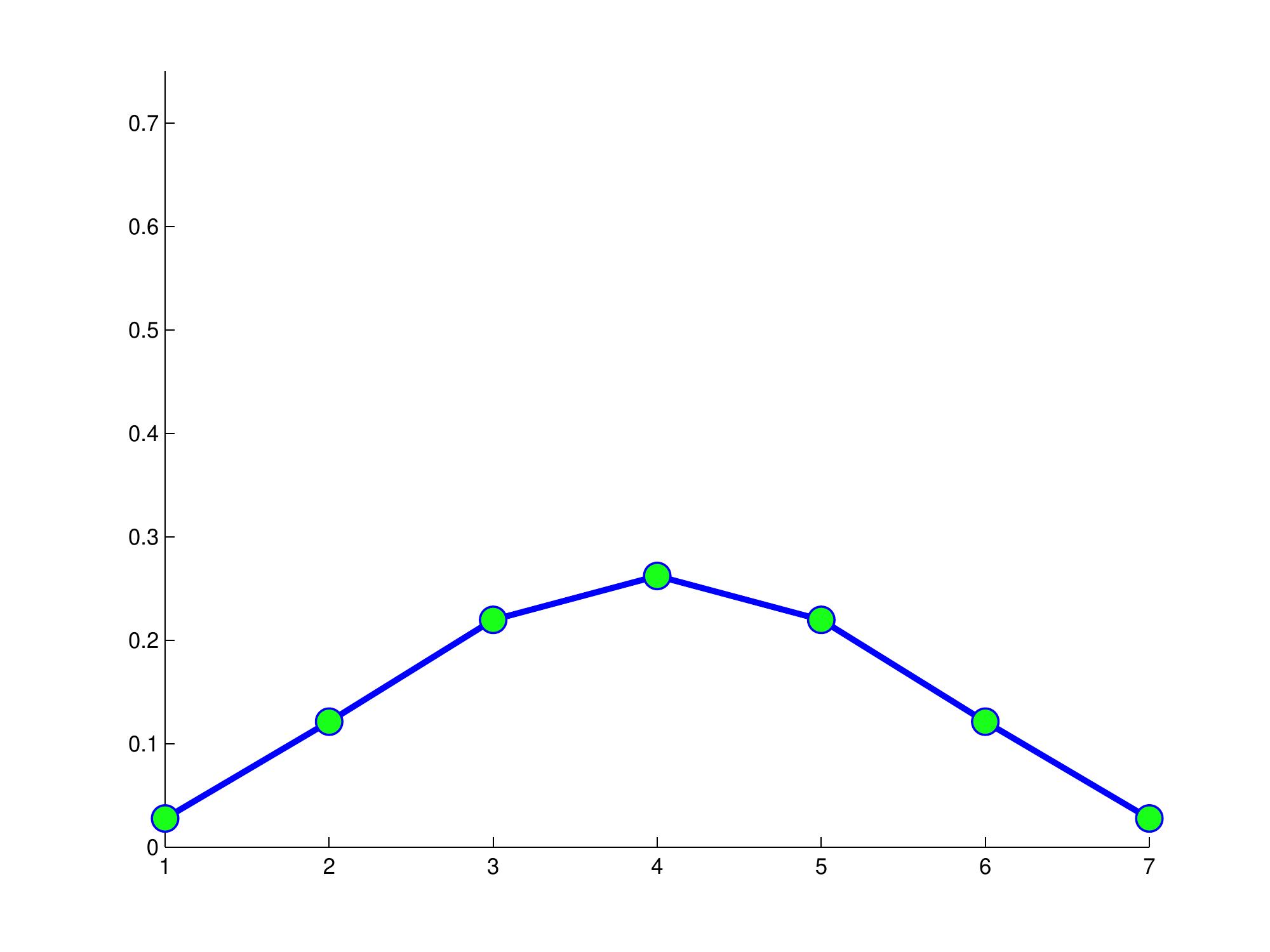}
 \put(38,21){$w_3$}
	\end{overpic}
 \hfill
 \vspace{-4pt}
\Acaption{1ex}{Top: Basis functions for the non-uniform knot sequences with five internal knots $\mathcal{X}_{6} = (x_0,x_1,...,x_6)$;
 each knot is of multiplicity two. Left: The internal knots are the roots of Chebyshev polynomial on $[-1,1]$ and Right:
 Legendre polynomial on [0,1]. The quadrature nodes $\tau_i$, $i=1,\dots, 7$ are shown in green.
 Bottom: The corresponding quadrature weights $\omega_i$, $i=1,\dots, 7$; in case of Chebyshev knots, the weights
 are normalized for the unit interval.}\label{fig:BasisChebyN=5}
 \end{figure}

   \begin{table}[!tbh]
 \begin{center}
  \begin{minipage}{0.9\textwidth}
\caption{Nodes and weights for particular knot sequences. $N$ denotes the number of internal knots. All the knots
and weights are normalized on unit interval and, due to the symmetry, only first $[\frac{N}{2}]+2$ nodes and weights are shown.}\label{tabW}
  \end{minipage}
\vspace{0.2cm}\\
\footnotesize{
\renewcommand{\arraystretch}{1.2}
\begin{tabular}{| c || r| r| r| r| r| r| r| r| r|}\hline
\rotatebox{0}{$N=5$}
 & \multicolumn{2}{c|}{Chebyshev}
 & \multicolumn{2}{c|}{Legendre}
 & \multicolumn{2}{c|}{Geometric $q=2$}\\
 $i$ & $\tau_i$ & $\omega_i$ & $\tau_i$ & $\omega_i$ & $\tau_i$ & $\omega_i$
\\\hline\hline
 1 & 0.006118 & 0.014502 & 0.011728 & 0.027799 & 0.017857 & 0.042328 \\\hline
 2 & 0.062790 & 0.113850 & 0.079882 & 0.121347 & 0.088993 & 0.104896 \\\hline
 3 & 0.233416 & 0.230297 & 0.251054 & 0.219793 & 0.244959 & 0.216881 \\\hline
 4 & 0.500000 & 0.282701 & 0.500000 & 0.262122 & 0.500000 & 0.271790 \\\hline\hline
 $N=6$ & \multicolumn{2}{c|}{} & \multicolumn{2}{c|}{} & \multicolumn{2}{c|}{}\\
 1 & 0.004259 & 0.010096 & 0.008441 & 0.020009 & 0.008333 & 0.019753 \\\hline
 2 & 0.044447 & 0.081009 & 0.058300 & 0.089278 & 0.041530 & 0.048952 \\\hline
 3 & 0.169161 & 0.172365 & 0.187089 & 0.169114 & 0.114314 & 0.101211 \\\hline
 4 & 0.378223 & 0.236530 & 0.386490 & 0.221598 & 0.312967 & 0.330084 \\\hline\hline
 $N=7$ & \multicolumn{2}{c|}{} & \multicolumn{2}{c|}{} & \multicolumn{2}{c|}{}\\
 1 & 0.003134 & 0.007429 & 0.006362 & 0.015079 & 0.008333 & 0.019753 \\\hline
 2 & 0.033034 & 0.060392 & 0.044320 & 0.068207 & 0.041530 & 0.048952 \\\hline
 3 & 0.127538 & 0.132404 & 0.144115 & 0.132816 & 0.114314 & 0.101211 \\\hline
 4 & 0.292314 & 0.192325 & 0.304385 & 0.183131 & 0.261560 & 0.203096 \\\hline
 5 & 0.500000 & 0.214901 & 0.500000 & 0.201532 & 0.500000 & 0.253977 \\\hline\hline
  $N=8$ & \multicolumn{2}{c|}{} & \multicolumn{2}{c|}{} & \multicolumn{2}{c|}{}\\
 1 & 0.002402 & 0.005693 & 0.004964 & 0.011766 & 0.004032 & 0.009558 \\\hline
 2 & 0.025481 & 0.046676 & 0.034784 & 0.053707 & 0.020095 & 0.023686 \\\hline
 3 & 0.099304 & 0.104319 & 0.114113 & 0.106506 & 0.055313 & 0.048973 \\\hline
 4 & 0.231216 & 0.156780 & 0.244557 & 0.151589 & 0.126561 & 0.098272 \\\hline
 4 & 0.405347 & 0.186531 & 0.410645 & 0.176432 & 0.318965 & 0.319511 \\\hline\hline
 $N=9$ & \multicolumn{2}{c|}{} & \multicolumn{2}{c|}{} & \multicolumn{2}{c|}{}\\
 1 & 0.001899 & 0.004501 & 0.003980 & 0.009434 & 0.004032 & 0.009558 \\\hline
 2 & 0.020237 & 0.037119 & 0.028004 & 0.043337 & 0.020095 & 0.023686 \\\hline
 3 & 0.079375 & 0.084052 & 0.092445 & 0.087039 & 0.055313 & 0.048973 \\\hline
 4 & 0.186823 & 0.129241 & 0.200155 & 0.126607 & 0.126561 & 0.098272 \\\hline
 5 & 0.332973 & 0.159838 & 0.341205 & 0.152710 & 0.269215 & 0.196605 \\\hline
 6 & 0.500000 & 0.170498 & 0.500000 & 0.161745 & 0.500000 & 0.245812 \\\hline
\end{tabular}
}
\end{center}
\end{table}

 \begin{figure}[!tbh]
 \hfill
 \begin{overpic}[width=.49\columnwidth,angle=0]{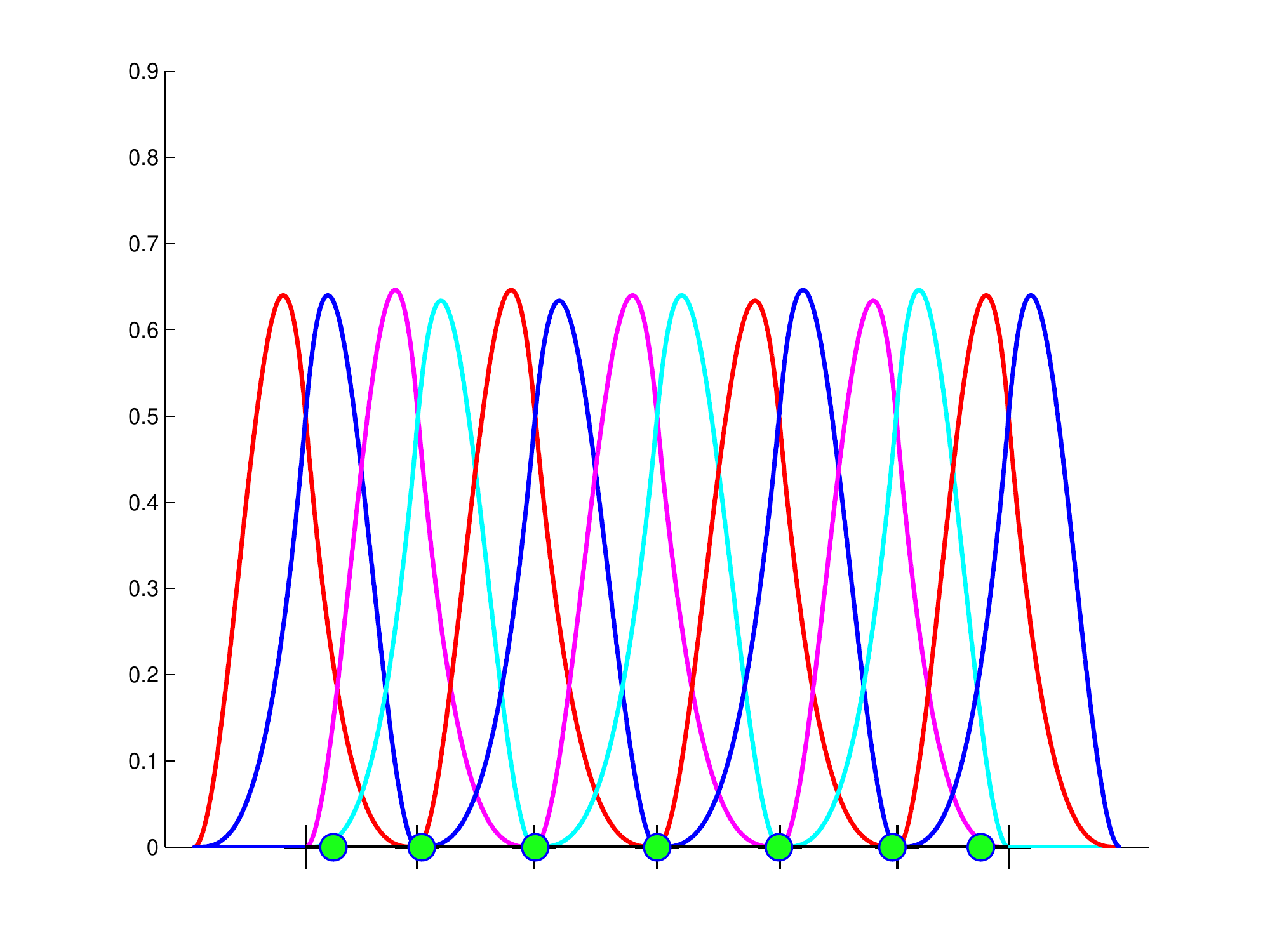}
 \put(38,55){\fcolorbox{gray}{white}{\includegraphics[width=0.115\textwidth]{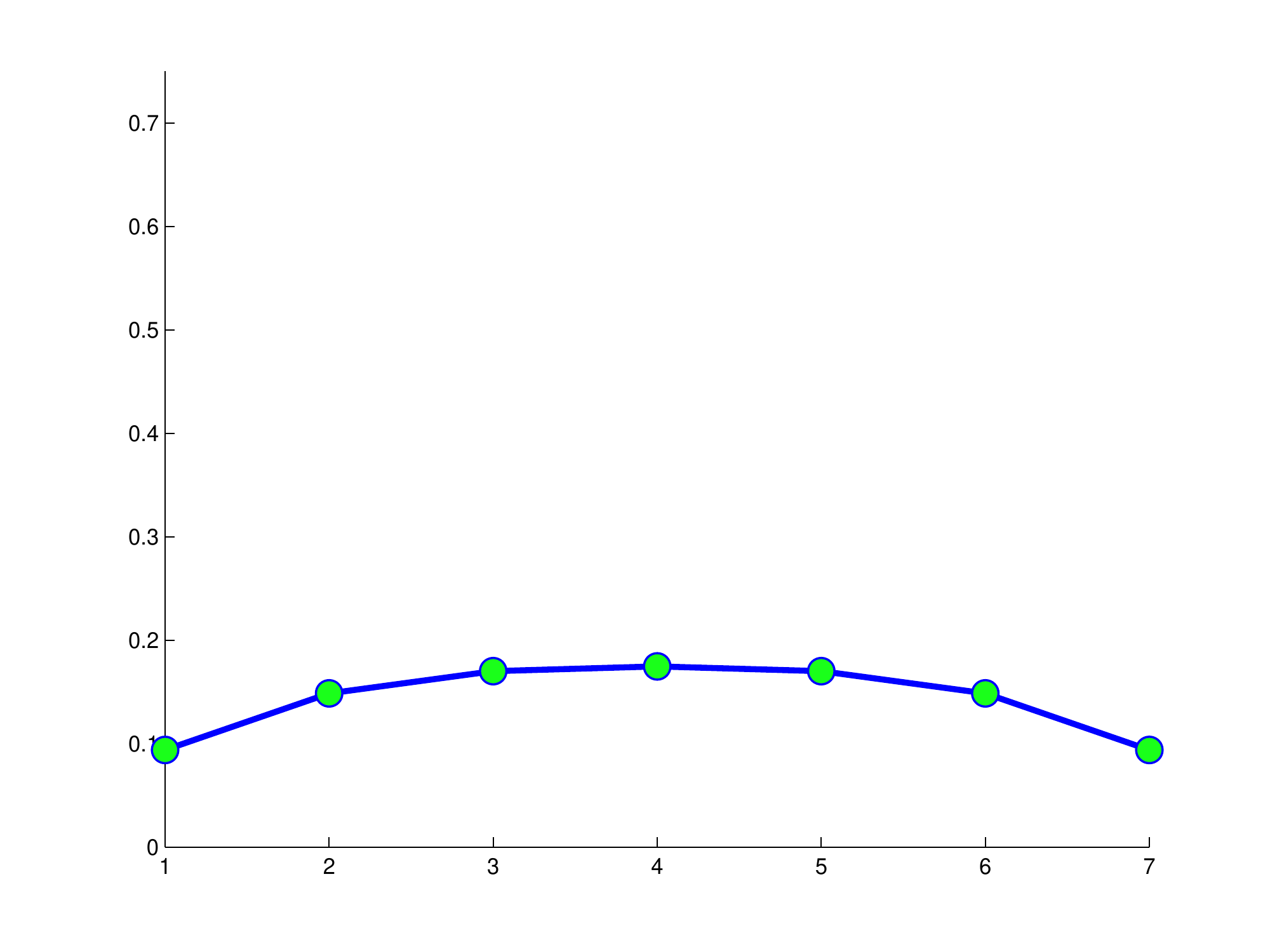}}}
    \put(15,65){$q=1.05$}
    \put(20,3){$x_0$}
    \put(30,3){$x_1$}
    \put(40,3){$x_2$}
	\end{overpic}
 \hfill
 \begin{overpic}[width=.49\columnwidth,angle=0]{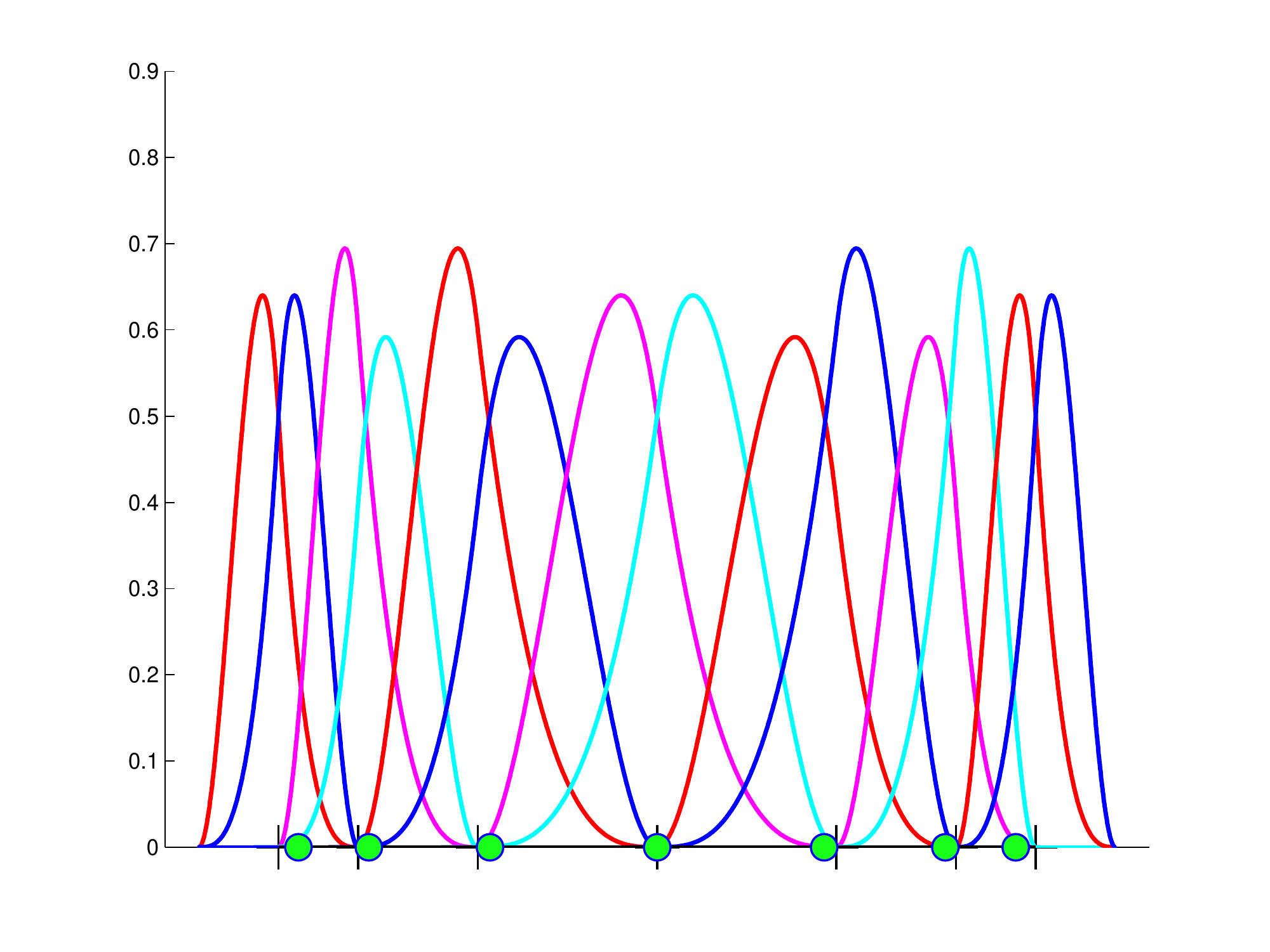}
    \put(38,55){\fcolorbox{gray}{white}{\includegraphics[width=0.115\textwidth]{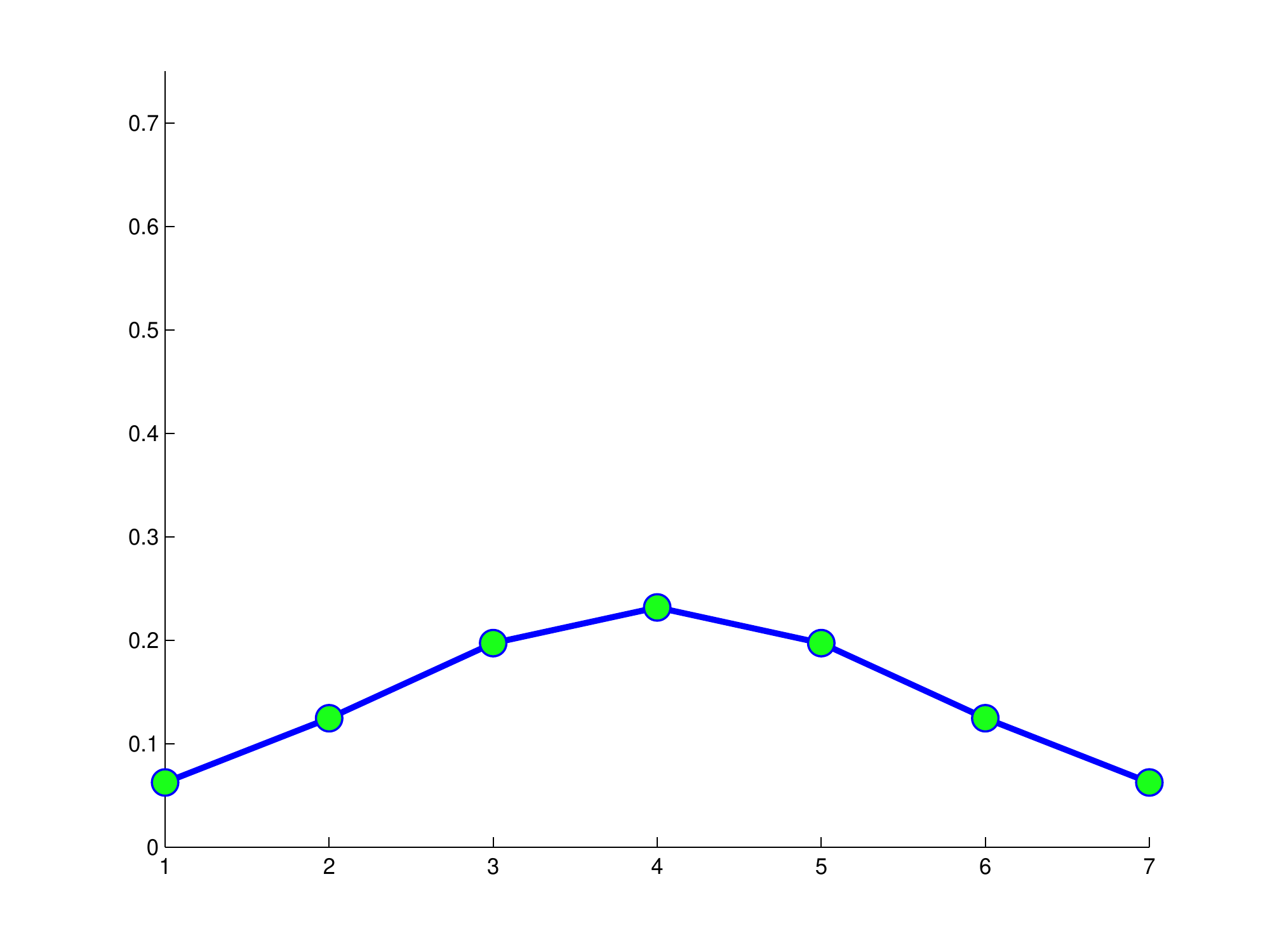}}}
    \put(15,65){$q=1.5$}
    \put(18,3){$x_0$}
    \put(26,3){$x_1$}
    \put(34,3){$x_2$}
	\end{overpic}
 \hfill
 \vspace{-0.01cm}
 \hfill
 \begin{overpic}[width=.49\columnwidth,angle=0]{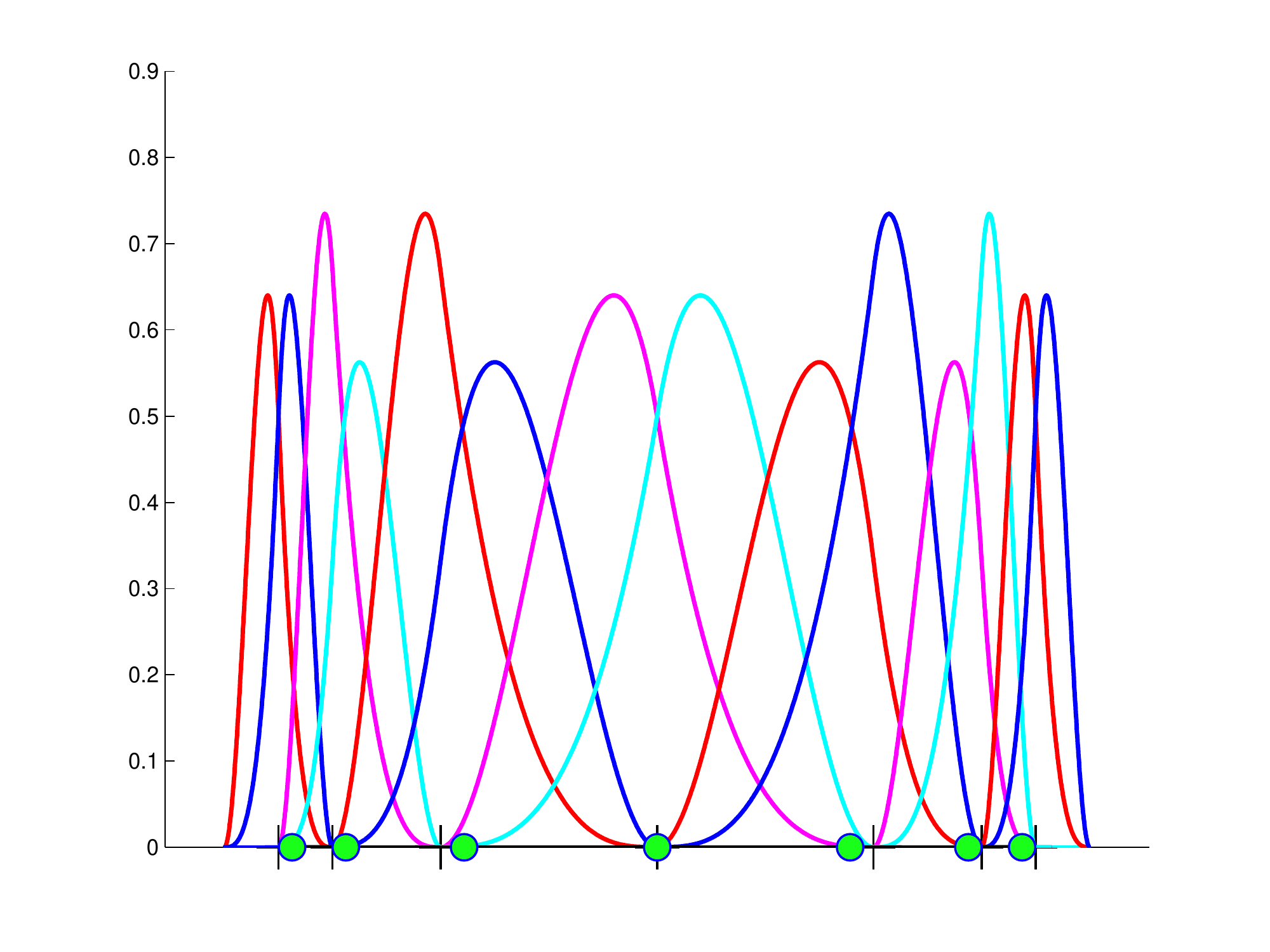}
 \put(38,55){\fcolorbox{gray}{white}{\includegraphics[width=0.115\textwidth]{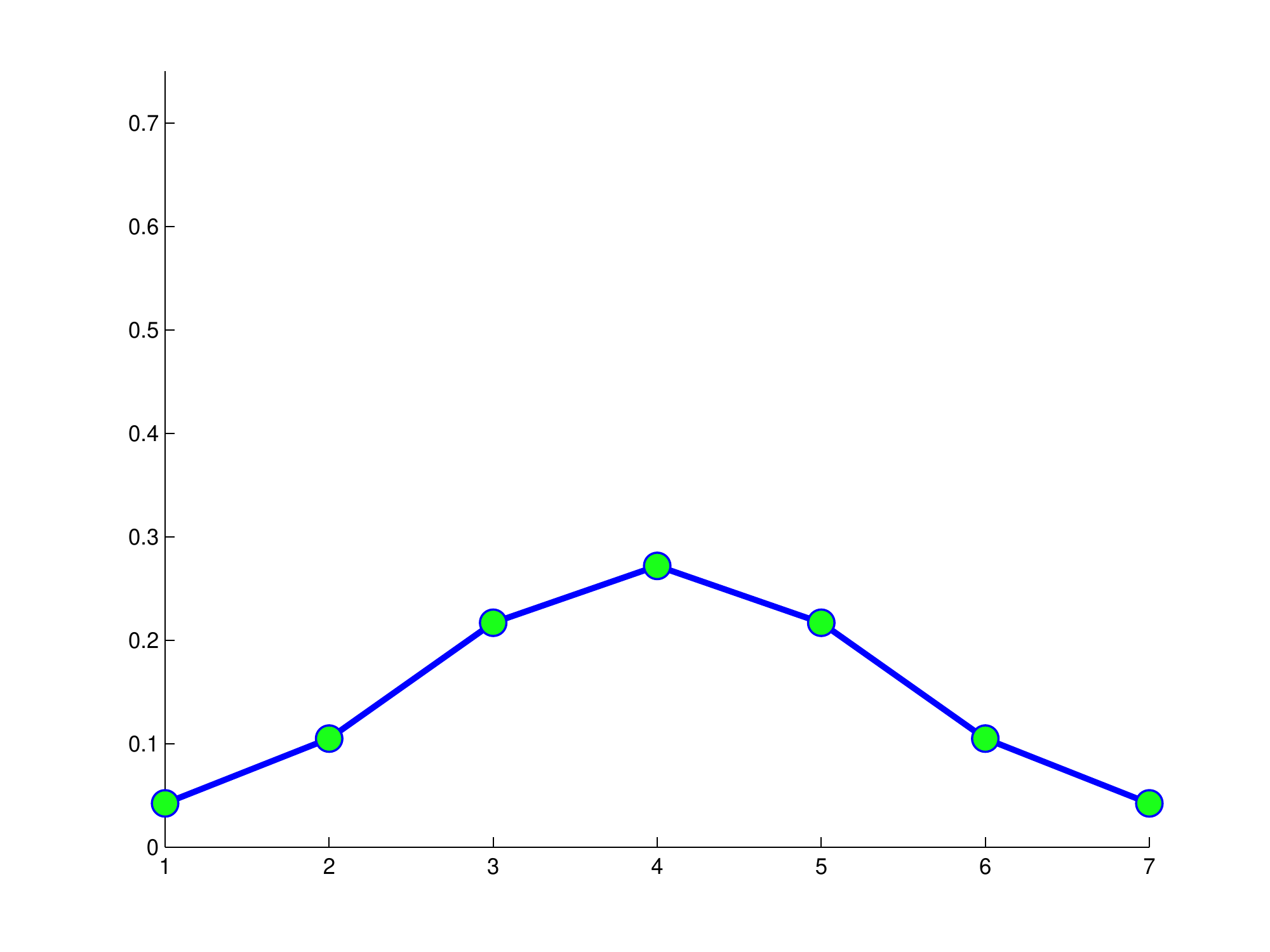}}}
 \put(15,65){$q=2$}
	\end{overpic}
 \hfill
 \begin{overpic}[width=.49\columnwidth,angle=0]{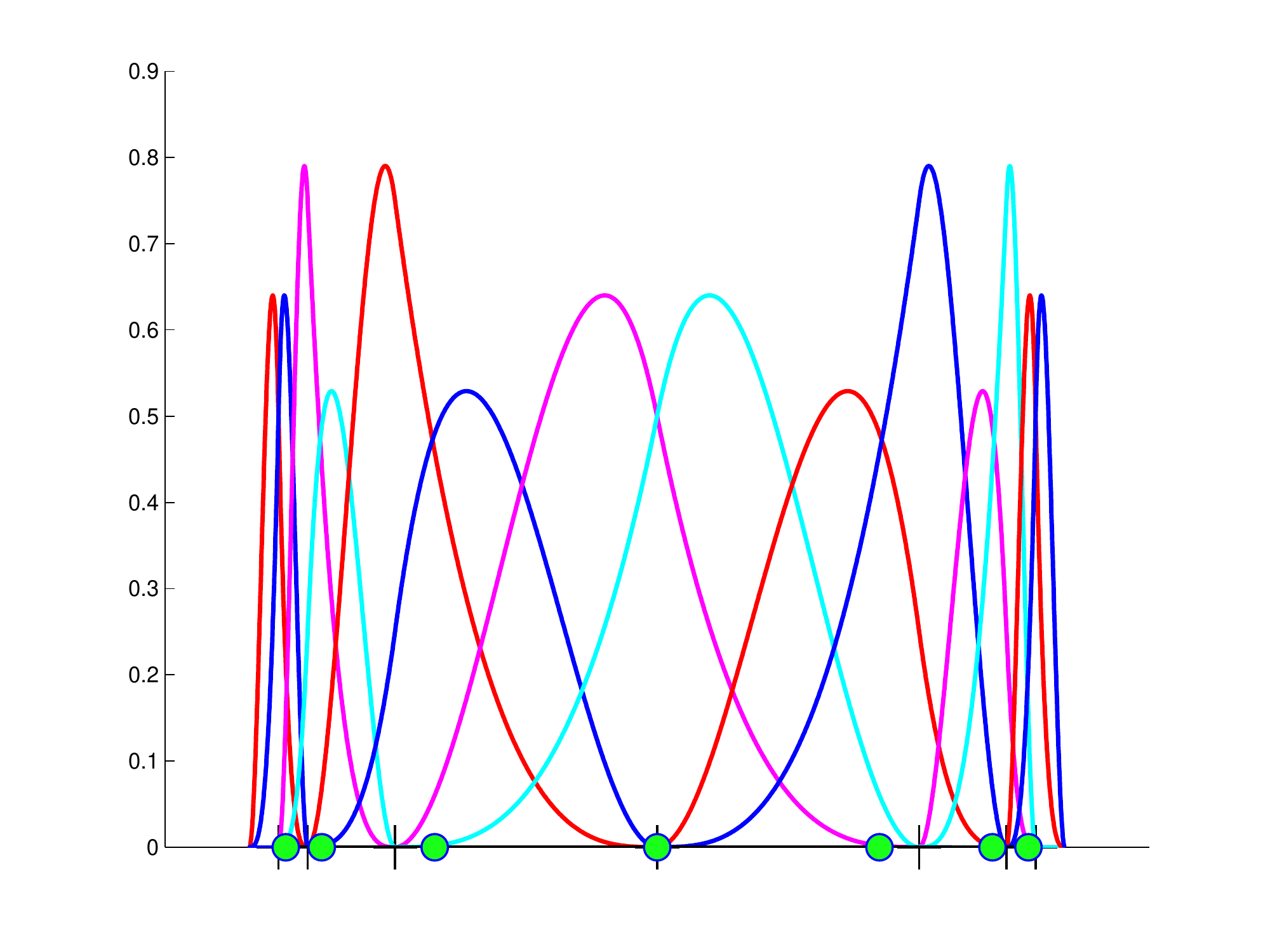}
    \put(38,55){\fcolorbox{gray}{white}{\includegraphics[width=0.115\textwidth]{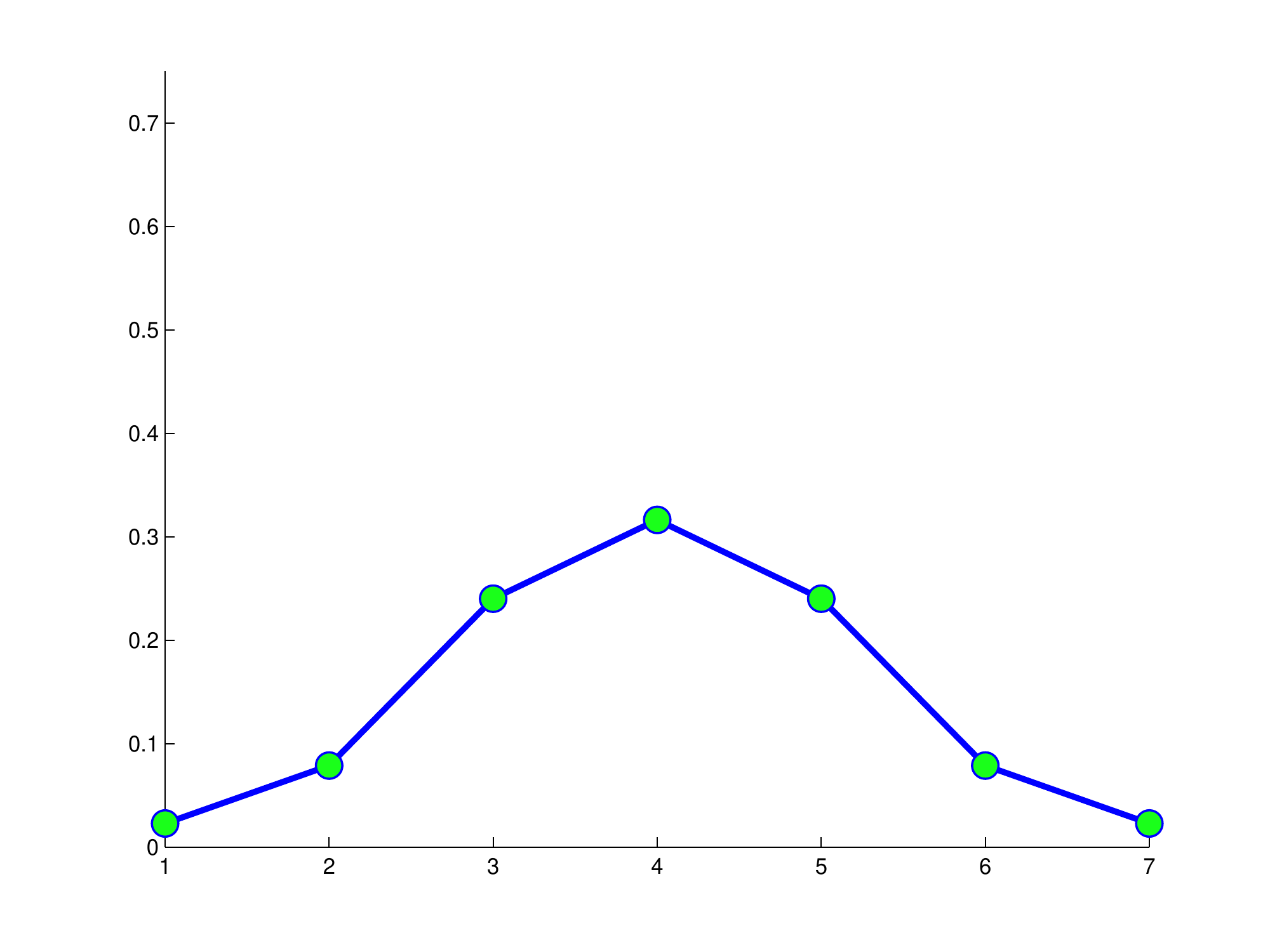}}}
    \put(15,65){$q=3$}
	\end{overpic}
 \hfill
 \vspace{-4pt}
 \Acaption{1ex}{For geometric knot sequences, the length of neighboring subintervals growth geometrically, i.e.
  $x_{k+1} - x_k = q (x_{k} - x_{k-1})$. The basis functions for a fixed number of internal knots ($N=5$) with various $q$ are shown.
  The green dots indicate the quadrature nodes and the top snapshot windows display their corresponding weights.}\label{fig:GeomN=5}
 \end{figure}

One such a prominent symmetrically stretched knot sequence stems from Chebyshev polynomials \cite{Gautschi-1997},
where its degree $N$ determines the roots which can be written as
\begin{equation}
x_k = - \cos(\phi_k), \quad \phi_k = \frac{2k-1}{2N}\pi, \quad k =1,2,\dots,N
\end{equation}
and the roots, according to Def.~\ref{def:stretch}, obviously form a non-uniform stretched knot
sequence on $[-1,1]$. The corresponding nodes and weights for $n-1=N=5$ are shown in Fig.~\ref{fig:BasisChebyN=5}.
Similarly, Legendre polynomials \cite{Szego-1936} satisfy the requirement that their roots form a
symmetrically stretched sequence. In order to have a qualitative comparison of the weights for
Chebyshev and Legendre knot sequences, and also for the comparison of their Peano kernels,
see Fig.~\ref{fig:PK}, the roots of Chebyshev polynomial were mapped to the unit domain.

Another family of symmetrically stretched knot sequences are those where the lengths of two neighboring knots form
a geometric sequence, i.e. the stretching ratio $q$ is constant, see Fig.~\ref{fig:GeomN=5}.
Obviously, the quadrature rule of Nikolov \cite{Nikolov-1996} is a special case for $q=1$. In some applications such as
solving the $1D$ heat equation \cite{Veerapaneni-2007} or simulating turbulent flows in 3D \cite{Bazilevs-2007,Bazilevs-2010},
where the finer and finer subdivisions closer to the domain boundary are needed, the uniform rule would eventually require large number
of knots whilst setting a proper non-uniform knot sequence could reduce the number of evaluations significantly.
The Peano kernels of geometric knot sequences considered as a function of the stretching ratio $q$
are shown in Fig.~\ref{fig:PKSrf}. It is not surprising, rather an expected result that the error constant $c_{n+1,4}$
looks favorably for the uniform knot sequence as the uniform layout is a certain equilibrium, that is,
a minimizer of the first term on the left side in (\ref{eq:c}).

We emphasize that these three types of non-uniform knot sequences are particular examples,
one can use any knot sequence satisfying Def.~\ref{def:stretch} that is suitable for a concrete application.
In all the numerical examples shown in the paper, we observed a similar phenomenon as in \cite{Nikolov-1996}, namely
that the weights are monotonically increasing when coming from the side to the middle of the interval, see Table~\ref{tabW}.
However, the proof for non-uniform knot sequences turned out to be rather difficult and we content ourselves here to
formulate it as an open problem, namely {\it{the quadrature nodes and weights computed in Theorem~\ref{theo:quad},
satisfy the inequalities}}
\begin{equation*}
\omega_i < \theta_i < \omega_{i+1} \quad \textnormal{for} \quad i=1,\dots, [\frac{n}{2}].
\end{equation*}

\begin{figure}[tbh!]
\hfill
 \begin{overpic}[width=.49\columnwidth,angle=0]{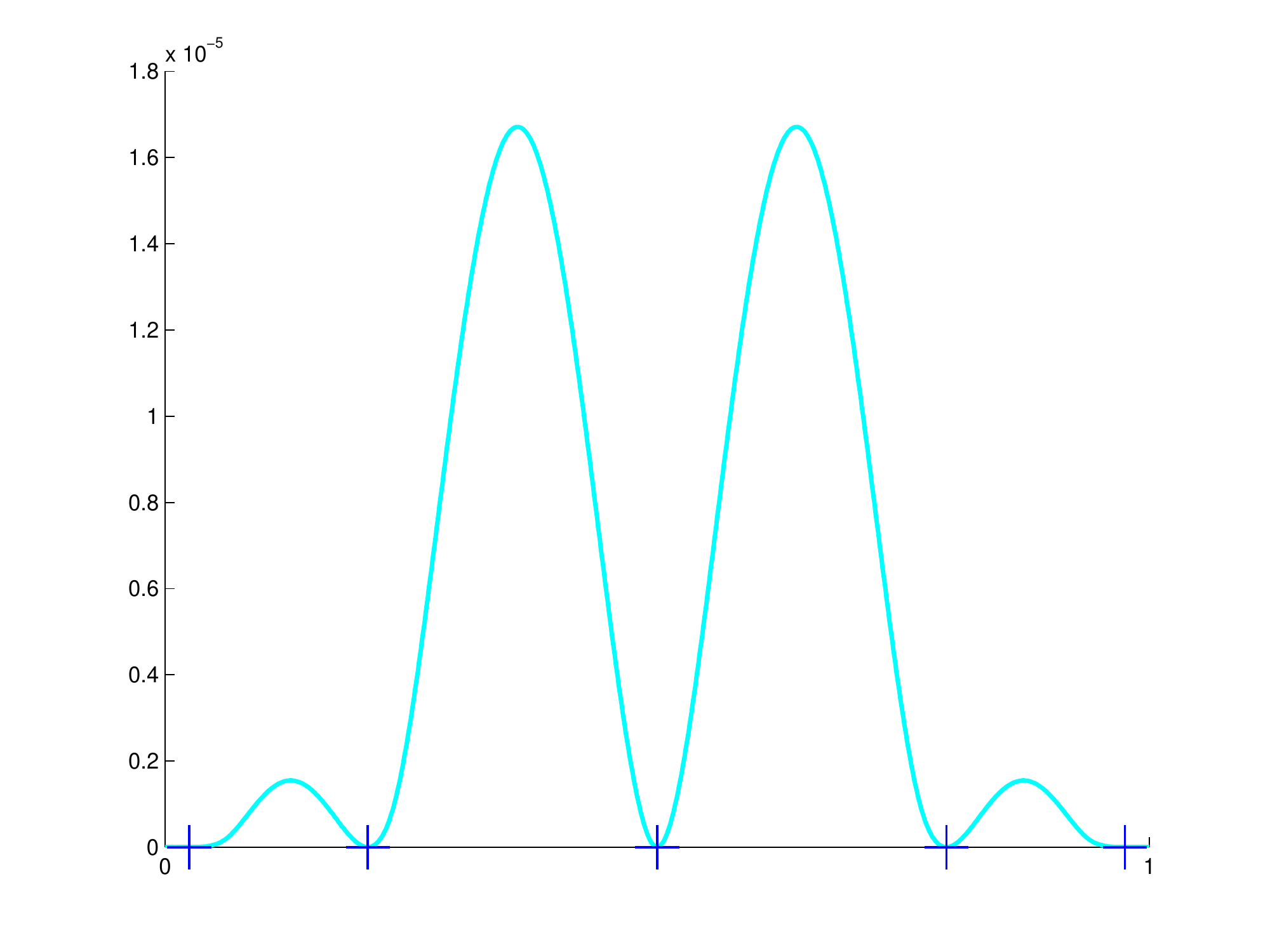}
    \put(13,2){$x_1$}
    \put(50,2){$x_3$}
    \put(85,2){$x_5$}
    \put(90,50){$N=5$}
    \put(40,70){\small Chebyshev}
	\end{overpic}
 \hfill
 \begin{overpic}[width=.49\columnwidth,angle=0]{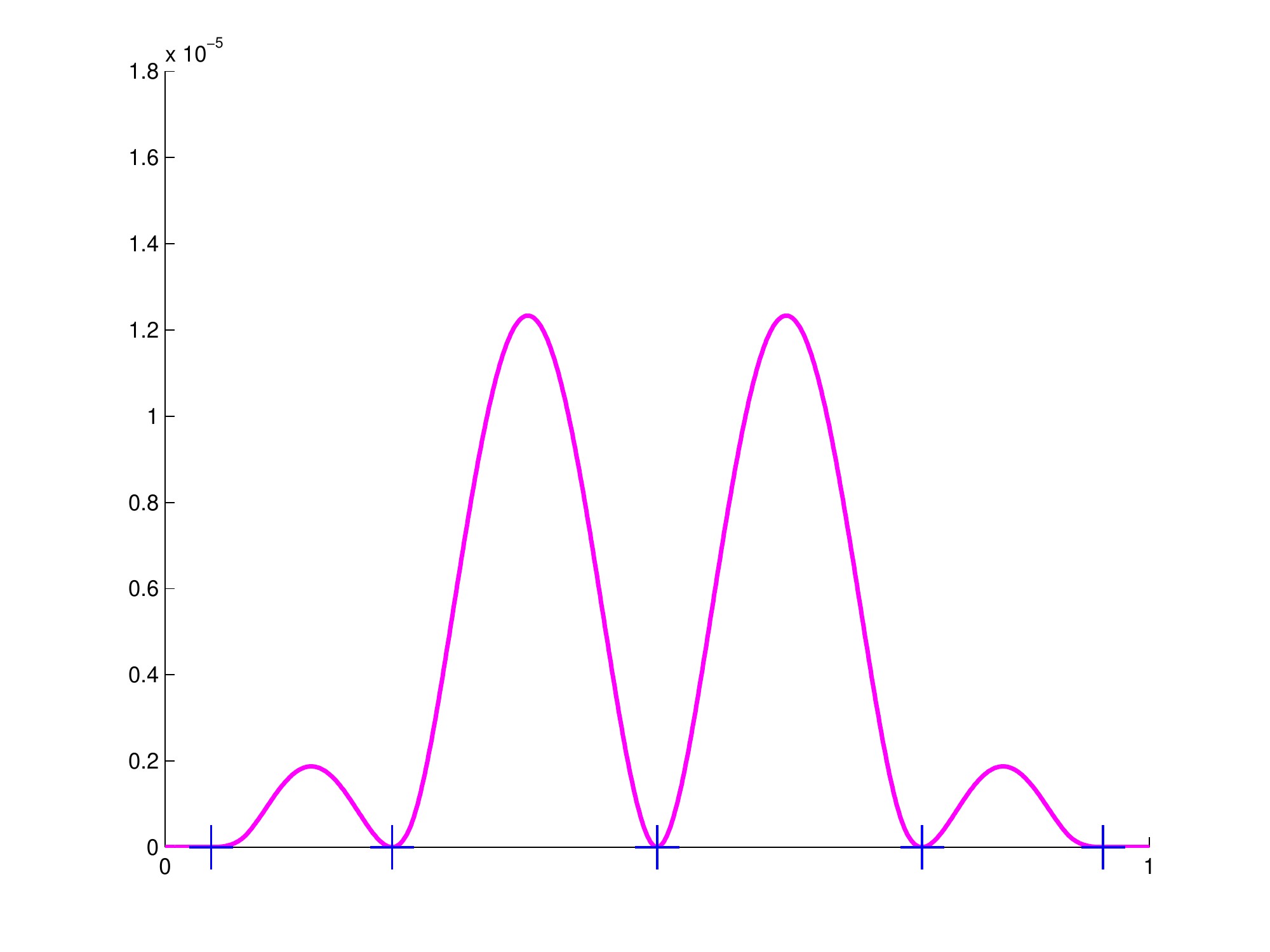}
    \put(40,70){\small Legendre}
	\end{overpic}
 \hfill
\vspace{-0.01cm}
 \hfill
 \begin{overpic}[width=.49\columnwidth,angle=0]{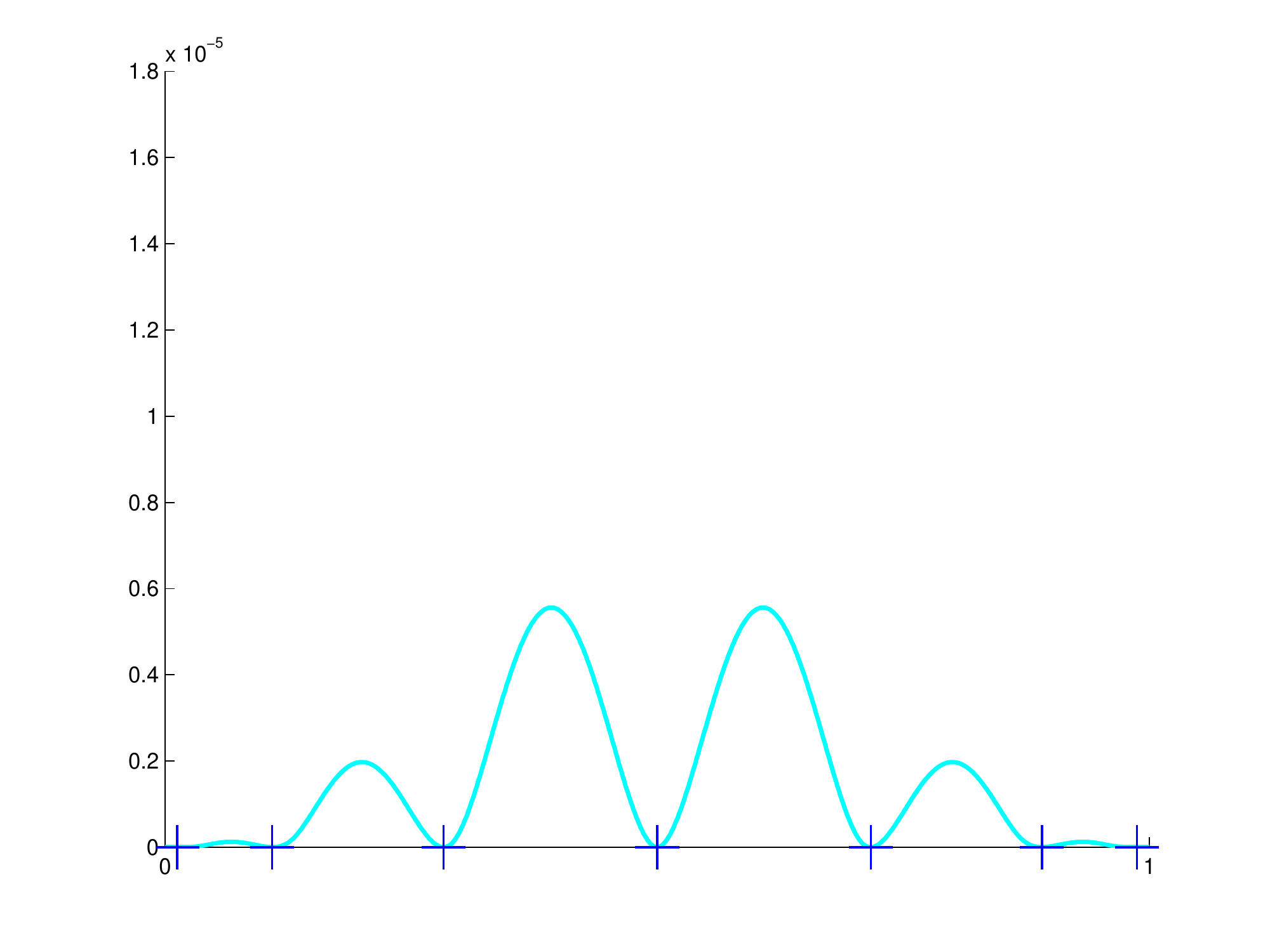}
    \put(90,50){$N=7$}
	\end{overpic}
 \hfill
 \begin{overpic}[width=.49\columnwidth,angle=0]{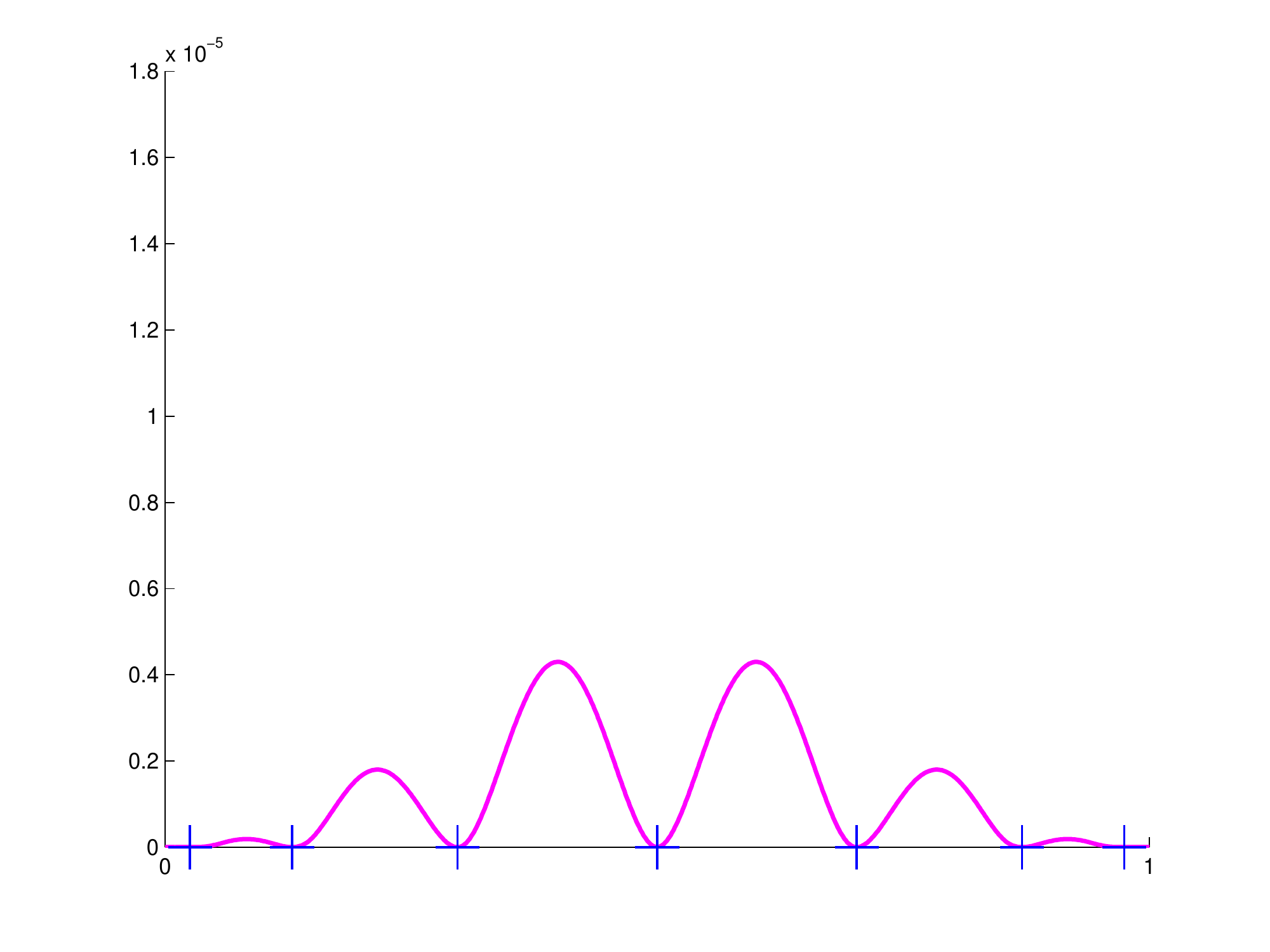}
	\end{overpic}
 \hfill
 \vspace{-4pt}
 \Acaption{1ex}{Peano kernels representing the constant $c_{n+1,4}$, see Eq.~(\ref{remainder}),
 for Chebyshev and Legendre knot sequences for $N=5$ and 7 on the unit domain are shown.}\label{fig:PK}
 \end{figure}

\begin{figure}[tbh!]
\hfill
 \begin{overpic}[width=.49\columnwidth,angle=0]{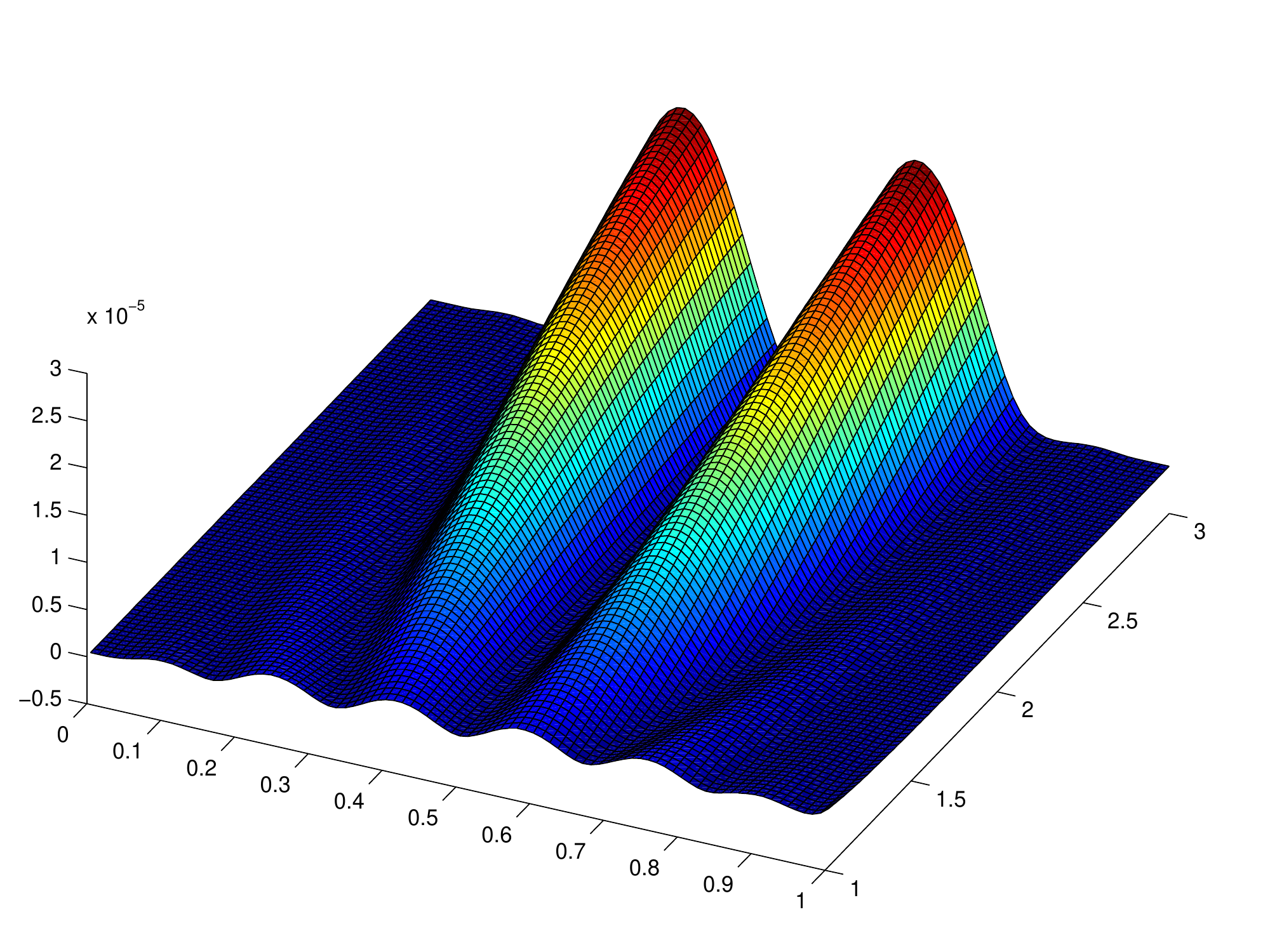}
    \put(10,55){$N=5$}
    \put(2,12){$a$}
    \put(60,-2){$b$}
    \put(70,2){$q=1$}
    \put(85,17){$q$}
    \put(90,65){\small Geometric}
	\end{overpic}
 \hfill
 \begin{overpic}[width=.49\columnwidth,angle=0]{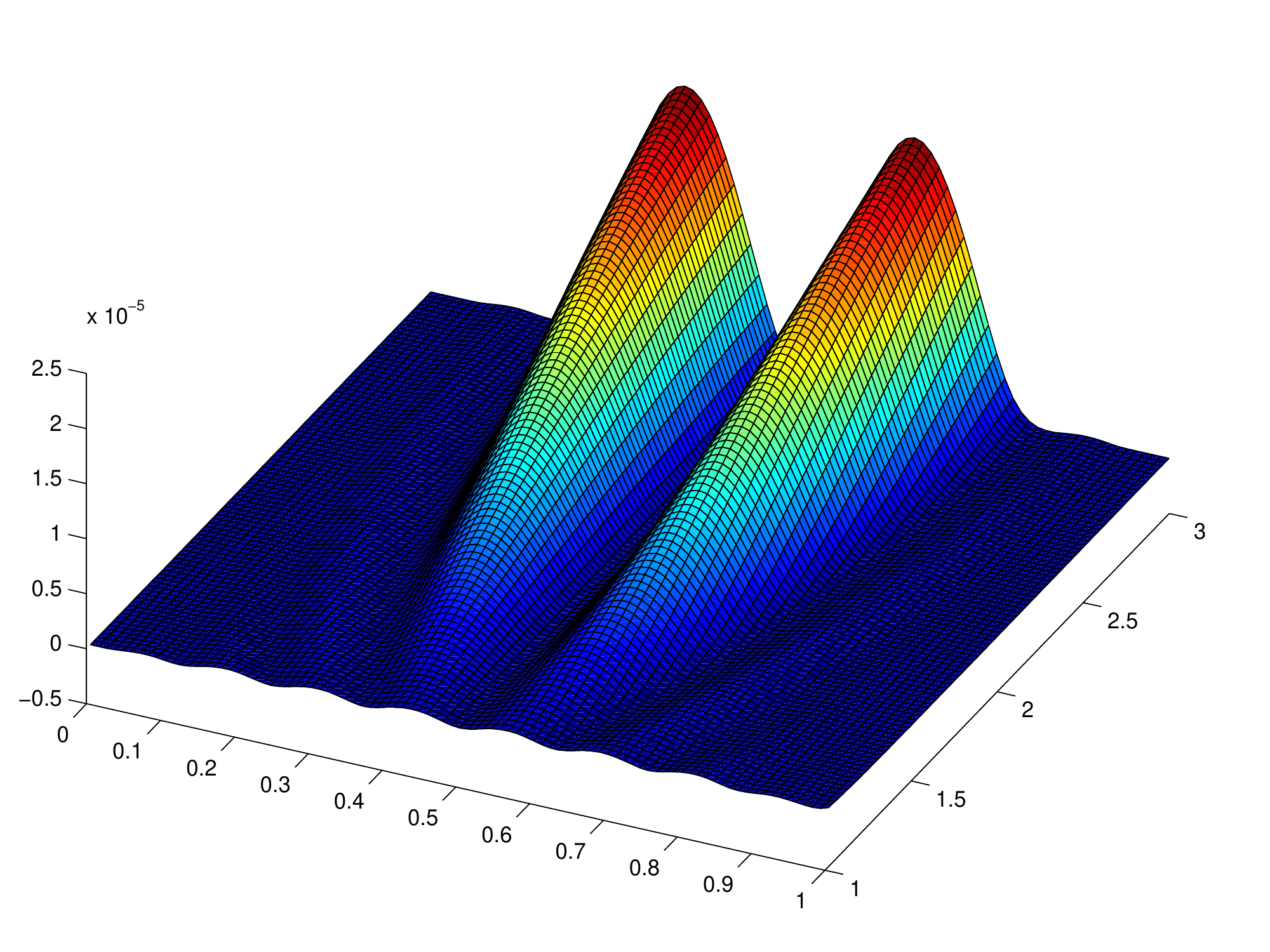}
    \put(10,55){$N=7$}
    \put(2,12){$a$}
    \put(60,-2){$b$}
    \put(70,2){$q=1$}
    \put(85,17){$q$}
	\end{overpic}
 \hfill
 \vspace{-4pt}
 \Acaption{1ex}{Peano kernels of a geometric knot sequence with $N=5$ and $7$ internal knots as a function of the scaling ratio $q$
 are shown. The cut by $q = const.$ plane is the corresponding univariate Peano kernel and its integral represents the error constant $c_{n+1,4}$,
 see (\ref{remainder}). The front boundary curve ($q=1$) is the Peano kernel
 associated to the uniform knot sequence.}\label{fig:PKSrf}
 \end{figure}

\section{Conclusion and future work}\label{sec:con} %

We have derived a quadrature rule for spaces of $C^1$ cubic splines with symmetrically
stretched knot sequences.
The rule possesses three crucial properties: We can exactly integrate the
functions from the space of interest; the rule requires minimal number of evaluations; and
the rule is defined in closed form, that is, we give explicit formulae without
need of any numerical algorithm.
To the best knowledge of the authors, the result is the first of the kind that handles
non-uniform knot sequences explicitly and, even though the symmetrical stretching seems
to be relatively restrictive, we believe that the infinite dimensional space of possible
knot sequences where the rule applies makes it a useful tool in many engineering applications.

Moreover, our quadrature rule is still exact, even though not optimal, for $C^2$ cubic splines.
Due to its explicitness,
it can also be freely used in various applications instead of $(3,2)$ splines quadrature rules,
for which the explicit formulae are not known.
In the future, we intent to derive quadrature rules for other spline spaces,
while aiming at particular engineering application.

\section*{Acknowledgments}
The research of the first author was supported by the KAUST Visual Computing Center.


\bibliographystyle{plain}
\bibliography{CubicC1QuadratureArXiv}

\end{document}